\documentclass[11pt]{article}
\usepackage{amsmath}
\usepackage{amsfonts}
\usepackage{graphicx}
\pdfoutput=1
\usepackage{indentfirst}
\usepackage{amsthm}
\usepackage[a4paper,body={180mm,260mm},centering]{geometry}

\newtheorem{lemma}{Lemma}[section]
\newtheorem{theorem}{Theorem}[section]

\allowdisplaybreaks[4]
\date{}
\begin{document}

\title{Dynamic Programming Principle for Backward Doubly Stochastic Recursive Optimal Control Problem and Sobolev Weak Solution of The Stochastic Hamilton-Bellman Equation
\thanks{This work is supported by National Natural Science Foundation of China (11501532 and 11301530),
        the Natural Science Foundation of Shandong Provence (ZR2015AQ004).}
\author{Yunhong Li$^a$ \hspace{1cm}Anis.Matoussi$^b$\hspace{1cm}Lifeng Wei$^c$\hspace{1cm}Zhen Wu$^d$\\
\small\em $^a$School of Mathematical Sciences, Ocean University of China,
\small\em Qingdao {\rm 266003}, P. R. China,\\
\small \em E-mail address:  liyunhong9970@stu.ouc.edu.cn (Y.Li)\\
\small\em $^b$University of Maine, Risk and Insurance Institute of Le Mans, Laboratoire Manceau de Math¨¦matiques,\\
\small\em Avenue Olivier Messiaen\\
\small \em E-mail address: anis.matoussi@univ-lemans.fr (A.Matoussi)\\
\small\em $^c$School of Mathematical Sciences, Ocean University of China,
\small\em Qingdao {\rm 266003}, P. R. China,\\
\small \em Corresponding author E-mail address: weilifeng@ouc.edu.cn (L.Wei)\\
\small\em $^d$Shandong University,
\small\em Jinan {\rm 250100}, P. R. China\\
\small  \em E-mail address: wuzhen@sdu.edu.cn (Z.Wu)
}}
\maketitle
\begin{abstract}
In this paper, we study backward doubly stochastic recursive optimal
control problem where the cost function is described by the
solution of a backward doubly stochastic differential equation. We
give the dynamical programming principle for this kind of optimal
control problem and show that the value function is the unique Sobolev weak solution for the
corresponding stochastic Hamilton-Jacobi-Bellman equation.
\end{abstract}
\vspace{2mm}


\section{Introduction}
Backward stochastic differential equation (BSDE in short) has
been introduced by Pardoux and Peng [3]. Independently, Duffie and
Epstein [2] introduced BSDE from economic background. In [2] they
presented a stochastic differential recursive utility which is an
extension of the standard additive utility with the instantaneous
utility depending not only on the instantaneous consumption rate but
also on the future utility. The recursive optimal control problem is
presented as a kind of optimal control problem whose cost functional
is described by the solution of BSDE. In [4] they gave the
formulation of recursive utilities and their properties from the
BSDE point of view. In 1992, Peng [6] got the Bellman's dynamic
programming principle for this kind of problem and proved that the
value function is a viscosity solution of one kind of quasi-linear
second-order partial differential equation (PDE in short) which is
the well-known as Hamilton-Jacobi-Bellman equation. Later in 1997£¬
he virtually generalized these results to a much more general
situation, under Markvian and even Non-Markvian framework. In this
chinese version, Peng used the backward semigroup property
introduced by a BSDE under Markovian and Non-Markovian framework. He
also proved that the value function is a viscosity solution of a
generalized Hamilton-Jacobi-Bellman equation. In 2007, Wu and Yu
[7] gave the dynamic programming principle for one kind of
stochastic recursive optimal control problem with the obstacle
constraint for the cost functional described by the solution of a
reflected BSDE and showed that the value function is the unique
viscosity solution of the obstacle problem for the corresponding
Hamilton-Jacobi-Bellman equation.

  In 1994, Pardoux and Peng first studied the backward doubly stochastic
differential equations(BDSDE in short). There are  two different
directions of stochastic integral in the equations involving with two independent standard Brownian motions: a
standard (forward) $dWt$ and a backward $dBt$. They had proved existence and uniqueness
result of this equation and established the connection
between BDSDE and a classical solution for stochastic partial
differential equation (SPDE in short) under  smoothness assumption
on the coefficients. And then, Bally and Matoussi [1] gave the probabilistic representation of the
solution in Sobolev space of semilinear stochastic PDE¡¯s in terms of BDSDE.
 Shi and Gu [16] gave the comparison theorem of BDSDE. Then Auguste and Modeste [10] got the uniqueness and existence of reflected BDSDE's solutions.

In our paper, we study a stochastic recursive optimal control problem where the control system is described by the classical stochastic differential equation, however, the cost function
is described by the solution of a backward doubly stochastic differential equation. This kind of recursive optimal control problem has some practical meaning.
For example,
in an arbitrage-free incomplete financial market, there may exist so called informal trading such as ``insider trading''.
An individual  has access to insider information would have an unfair edge over other investors, who do not have the same access, and could potentially make larger `unfair' profits than their fellow investors.
This phenomenon could be described by a BDSDE in a financial market models.
More specifically, there are two kinds of investors with different levels of information about the future price evolution in a market influenced by an additional source of randomness.
The ordinary trader only has the ``public information''---market prices of the underlying assets contained in the filtration $\mathcal{F}_{t}^{W}$. However, an insider who has assess to a larger filtration $\mathcal{F}_{t}^{W} \vee \mathcal{F}_{t, T}^{B}$, which includes insider information. For instance, an insider knows the functional law of the price process or he knows in advance that a significant change has occurred in the business policy or scope of a security issue or he could estimate if his portfolio is better than others.
We would like to emphasize that BDSDE techniques provide powerful instruments to analyze the problem  of portfolio  optimization of an insider trader. For an insider trader, his investment strategy still satisfies the property that locally optimal is equal to  globally optimal.

The problem we are most interested in is whether the dynamic programming principle still holds for this recursive optimal control problem.
The good news is that it can be accomplished by the properties of the BDSDE. Compared with the Hamilton-Jacobi-Bellman (HJB in short) equation in paper[6][7],  the corresponding HJB we get is a SPDE in a Markovian framework.
In the stochastic case where the diffusion is possibly degenerate, the HJB equation may in general have no classical solution. To overcome this difficulty, Crandall and Lions introduced the so-called viscosity solutions in the early 1980s. Obviously, the research on the viscosity solution on HJB equations have yielded fruitful results. However, the viscosity solution of the HJB equation cannot give an reasonable probabilistic interpretation on a pair of solution (Y,Z) of BSDE considering that relationship do not established between the Z part of the solution and the HJB equation. Here, we study a different kind of weak solution for HJB equations in a Sobolev space, in which part Z is spontaneously contained in the weak definition. Wei and Wu [11] have proved that the value function is the unique Sobolev weak solution of the related HJB equation by virtue of the nonlinear Doob-Meyer decomposition theorem introduced in the study of BSDEs.

  In this paper, we consider the issue on Sobolev weak solution of HJB equation connected with BDSDE. Since that we cannot find a Doob-Meyer decomposition theorem in BDSDE, it is a point that how to establish the equation like Lemma 4.1. and 4.2. in \cite{WWZ}. 
   Inspired by the \cite{AM2007} £¬we bring a increasing process into the equation in order to push the cost functional upward in a minimum force.

The paper is organized as follows. Preliminaries and assumption are introduced in Section 2. In Section 3 we formulate a stochastic recursive optimal control problem where the cost function is
described by the solution of a BDSDE. We show that the celebrated dynamic
programming principle still holds for this kind of optimization
problem. In Section 4 we prove that the value function of this
problem is the unique weak solution in a Sobolev space for the
corresponding stochastic Hamilton-Jacobi-Bellman equation.

\vspace{2mm}

\section{Preliminaries and assumption}
\setlength{\parindent}{2em}In this section, we give some preliminary results of the BDSDE
which are useful for the dynamic programming principle for the
recursive optimal control problem.

\setlength{\parindent}{2em}Let $(\Omega, \mathcal{F}, \mathcal{P})$ be a probability space, and $T>0$ be an arbitrarily fixed constant throughout this paper. Let $\left\{W_{t} ; 0 \leq t \leq T\right\}$ and $\left\{B_{t} ; 0 \leq t \leq T\right\}$ be two mutually independent standard Brownian Motion processes with values respectively in $R^{d}$ and $R^{l}$, defined on $(\Omega, \mathcal{F}, \mathcal{P})$. Let $\mathcal{N}$ denote the class of $P$ - null sets of $\mathcal{F}$. For each $t \in[0, T]$, we define
$$
\mathcal{F}_{t}:=\mathcal{F}_{t}^{W} \vee \mathcal{F}_{t, T}^{B}\,,
$$
where for any process $\left\{\eta_{t}\right\}, \mathcal{F}_{s, t}^{\eta}=\sigma\left\{\eta_{r}-\eta_{s} ; s \leq r \leq t\right\} \vee \mathcal{N}, \mathcal{F}_{t}^{\eta}=\mathcal{F}_{0, t}^{\eta}$. Let $\mathcal{F}_{t^{\prime}}^{t}$ be the complete filtration generated by the Brownian motion $W_{t^{\prime}}-W_{t},$ so $\mathcal{F}_{t^{\prime}}^{t}=\sigma\left\{W_{r}-W_{t}, t \leq r \leq t^{\prime}\right\} \vee \mathcal{N}$.\\
Note that the collection $\left\{\mathcal{F}_{t} ; t \in[0, T]\right\}$ is neither increasing nor decreasing, so it does not constitute a filtration.\\
   Let us introduce some notations:
\vspace{2mm}

$\begin{array}{rl} \mathcal L^{p}&=\{\xi \textrm
{\,is\,an\,}\mathcal F_{T}-\textrm
                {measurable\,random\,variable}\,s.t.\,E(|\xi|^{p}) <+\infty,\quad p\geq2\},\\
\mathcal H^{p}&=\left \{\{\psi_{t},\ 0\leq t\leq T\}\  \textrm {is\
                a\ predictable\ process\ s.t.}\ E(\int^T_0|\varphi_{t}|^{2}dt)^{\frac{p}{2}}<+\infty,\quad p\geq2\right\},\\
\mathcal S^{p}&=\left\{\{\varphi_{t},\,0\leq t\leq T\}\,\textrm
                {is\,a\,predictable\,process\,s.t.}\,E(\sup\limits _{0\leq t\leq T}|\varphi_{t}|^{p})<+\infty,\quad p\geq2\right\}.
\end{array}$

and the following BDSDE:
\begin{eqnarray}
Y_{t}=\xi+\int^T_tf(s,Y_{s},Z_{s})ds+\int^T_tg(s,Y_{s},Z_{s})d{B}_{s}-\int^T_tZ_{s}dW_{s},\quad
0\leq t\leq T.
\end{eqnarray}
Here
$$
f:\Omega\times[0,T]\times\mathbb{R}^{k}\times\mathbb{R}^{k\times
d}\rightarrow \mathbb{R}^{k}\hspace{2mm},
$$
$$g:\Omega\times[0,T]\times\mathbb{R}^{k}\times\mathbb{R}^{k\times
d}\rightarrow \mathbb{R}^{k\times l},
$$
and $f,g$ satisfying\\
(H1)\,for all $(y,z)\in\mathbb{R}\times\mathbb{R}^{d}$,
$$f(\cdot,y,z)\in M^{P}(0,T;\mathbb{R}^{k});\,
g(\cdot,y,z)\in M^{P}(0,T;\mathbb{R}^{k\times l}),$$
 (H2)\,for some
$L>0$ and $0<\alpha<1$ all
$y,y'\in\mathbb{R},z,z'\in\mathbb{R}^{d},a.s.$
$$
|f(t,y,z)-f(t,y',z')|^2\leq L(|y-y'|^2+\|z-z'\|^2),\quad
\|g(t,y,z)-g(t,y',z')\|\leq L|y-y'|^2+\alpha\|z-z'\|^2.
$$
 There exists $C$ such that for all $(t,x,y,z,v)\in[o,T]\times R^d\times R^{k\times
 d}$,
 $$gg^\ast(t,x,y,z)\leq zz^\ast+C(\|g(t,x,o,o)\|^2+|y|^2)I.$$
(H3)\,$\xi\in L^{p}$.

  We notice that there are two independent Brownian motions $W$ and
$B$ in (1), where the $dW$ integral is a formed It\^{o}'s integral
and $dB$ integral is a backward It\^{o}'s integral. The extra noise
$B$ in the equation can be thought of as some extra information that
can not be detected in the market in general, but is available to
the particular investor. The problem then is to show how this
investor can take advantage of such extra information to optimize
the utility, but by taking actions that are completely ``legal", in
the sense that the investor has to choose the optimal strategy in
the usual class of the admissible portfolios.

Then form Theorem 1.1 in [5], then there exists a unique solution
$\{(Y_{t},Z_{t}),0\leq t\leq T\}\in\mathcal
{S}^{p}(0,T;\mathbb{R}^{k})\times\mathcal
{H}^{p}(0,T;\mathbb{R}^{k\times d})$.

Now we give two more accurate estimates of the solutions. They are
very important and necessary for the dynamic programming principle
of our optimal control problem and play an important role for the
continuation properties of value function $u(t,x)$ about $t$ and
$x$. The proof is complicated and technical, some technique derive
from [1]. \vspace{2mm}

{\bf Proposition 2.1}\ \quad{\it Let$\{(Y_{t},Z_{t})\in\mathcal
{S}^{p}(0,T;\mathbb{R}^{k})\times\mathcal
{H}^{p}(0,T;\mathbb{R}^{k\times d}),\quad 0\leq t\leq T\}$ be the
solution of the above
 BDSDE, then for some $p>2$, $\xi\in L^{p}(\Omega,\mathcal {F}_{T},T;\mathbb{R}^{k})$ and
 $$ E\int^T_0(|f(t,0,0)|^{p}+\|g(t,0,0)\|^{p})dt<\infty, $$
 we have}
\begin{eqnarray}
E\left \{\sup_{t\leq s\leq T}|Y_{s}|^{p}+(\int^T_0
\|Z_{s}\|^{2}dt)^{\frac{p}{2}}\right \}<\infty.
\end{eqnarray}

\vspace{1mm} {\bf Proposition 2.2}\quad {\it Let $(\xi,f,g)$  and
$(\xi',f',g')$ be two triplets satisfying the above assumption.
Suppose $(Y,Z)$ is the solution of the BDSDE $(\xi,f,g)$ and
$(Y',Z')$ is the solution of the BDSDE $(\xi',f',g')$. Define
$$\triangle \xi=\xi-\xi',\quad \triangle f=f-f',\quad \triangle g=g-g',$$
$$\triangle Y=Y-Y',\quad \triangle Z=Z-Z'.$$
Then there exists a constant $C$ such that}

\begin{eqnarray}
E\left \{\sup_{t\leq s\leq T}|\triangle Y_{s}|^{p}+(\int^T_0
|\triangle Z_{s}|^{2}dt)^{\frac{p}{2}}\right \}\leq
CE\left\{|\triangle \xi|^{p}\right\}.
\end{eqnarray}

\section{Formulation of the problem and the Dynamic Programming
Principle}

\setlength{\parindent}{2em}In this section, we first formulate a backward doubly
stochastic recursive optimal control problem, and then we prove that
the dynamic programming principle still holds for this kind of
optimization problem.

\setlength{\parindent}{2em}We introduce the admissible control set $\mathcal {U}$
defined by
$$
\mathcal {U}:=\{v(\cdot)\in\mathcal{H}^{p}|\,v(\cdot)\,\textrm{take
value in}\,U\subset\mathbb{R}^{k}\}.
$$
\vspace{1mm}An element of $\mathcal {U}$ is called an admissible
control. Here $\mathcal {U}$ is a compact subset of
$\mathbb{R}^{k}$, however this restriction is often satisfied in
practical applications.

\vspace{2mm}For a given admissible control, we consider the
following control system

\begin{equation}
\left\{\begin{array}{ll}
dX_{s}^{t,\zeta;v}=b(s,X_{s}^{t,\zeta;v},v_{s})ds+\sigma(s,X_{s}^{t,\zeta;v},v_{s})dW_{s},\quad
s\in[t,
T],\\
X_{t}^{t,\zeta;v}=\zeta,
\end{array}\right.
\end{equation}

  Where $t\geq0$ is regarded as the initial time and $\zeta\in\mathcal
{L}^{p}(\Omega,\mathcal {F}_{t},P;\mathbb{R}^{n})$ as the initial
state. The mappings
$$
b:[0,T]\times\mathbb{R}^{n}\times\mathbb{U}\longrightarrow\mathbb{R}^{n},\qquad\sigma:[0,T]\times\mathbb{R}^{n}\times\mathbb{U}\longrightarrow\mathbb{R}^{n\times
d}.
$$

satisfy the following conditions:

(H3.1) $b$ and $\sigma$ are continuous in t.

(H3.2) for some $L>0$,and all
$x,x^{'}\in\mathbb{R}^{n},v,v^{'}\in\mathbb{U},a.s.$
$$
|b(t,x,v)-b(t^{'},x^{'},v^{'})|+|\sigma(t,x,v)-\sigma(t,x^{'},v^{'})|\leq
L(|x-x^{'}|+|v-v^{'}|).
$$

  Obviously, under the above assumption, for any $v(\cdot)\in\mathcal
{U}$, control system (4) has a unique strong solution
$\{X_{s}^{t,\zeta;v}\in\mathcal {H}^{p}(0,T;\mathbb{R}^{k}),0\leq
t\leq s\leq T\}$ , and we also have the following estimates.

 \vspace{2mm} {\bf Proposition 3.1}\ \quad{\it For all
 $t\in[0,T]$, $\zeta,\,\zeta'\in L^{p}(\Omega,\mathcal
 {F}_{t},P;\mathbb{R}^{n})$, $v(\cdot),\,v^{'}(\cdot)\in\mathcal {U}$,
\begin{eqnarray}
\mathbb{E}^{\mathcal {F}^{W}_{t}}\left\{\sup_{t\leq s\leq
T}|X_{s}^{t,\zeta;v}|^{p}\right\}\leq C_{p}(1+|\zeta|^{p}),
\end{eqnarray}
\begin{eqnarray}
\mathbb{E}^{\mathcal {F}^{W}_{t}}\left\{\sup_{t\leq s\leq
T}|X_{s}^{t,\zeta;v}-X_{s}^{t,\zeta';v'}|^{p}\right\}\leq
C_{p}|\zeta-\zeta'|^{p}+ C\mathbb{E}^{\mathcal
{F}^{W}_{t}}\left\{\int^{T}_{t}|v_{s}-v^{'}_{s}|^{p}ds\right\}.
\end{eqnarray}

  Where the constant $C_{p}$ also depends on $L$.}

 \vspace{2mm} {\bf Proposition 3.2}\ \quad{\it For all
 $t\in[0,T]$, $x\in\mathbb{R}^{n}$, $v(\cdot)\in\mathcal
 {U}$, $\delta\in[0,T-t]$,

\begin{eqnarray}
\mathbb{E}\left\{\sup_{t\leq s\leq
t+\delta}|X_{s}^{t,\zeta;v}-x|^{p}\right\}\leq
C_{P}\delta^{\frac{p}{2}},
\end{eqnarray}

  where the constant $C$ also depends on $x$ and $L$}.

\vspace{1mm}Now for any given admissible control $v(\cdot)$, we
consider the following BDSDE:

\begin{equation}
\begin{array}{lcl}
Y^{t,\zeta;v}_{s}&=&\Phi(X^{t,\zeta;v}_{T})+\int^T_s
f(r,X^{t,\zeta;v}_{r},Y^{t,\zeta;v}_{r},Z^{t,\zeta;v}_{r},v_{r})dr\\
&+& \int^T_s
g(r,X^{t,\zeta;v}_{r},Y^{t,\zeta;v}_{r},Z^{t,\zeta;v}_{r})d{B_{r}}-\int^T_s
Z^{t,\zeta;v}_{r}dW_{r},  \quad t\leq s\leq T,
\end{array}
\end{equation}

where

\begin{eqnarray*}
&&\Phi:\mathbb{R}^{n}\rightarrow\mathbb{R}^{n},\hspace{44mm}\\
&&f:\Omega\times[0,T]\times\mathbb{R}^{k}\times\mathbb{R}^{k\times
d}\times \mathcal {U}\rightarrow \mathbb{R}^{n},\hspace{3mm}\\
&&g:\Omega\times[0,T]\times\mathbb{R}^{k}\times\mathbb{R}^{k\times
d}\rightarrow \mathbb{R}^{n\times l},\\
\end{eqnarray*}

  and they satisfy the following conditions:

(H3.3) $f$ and $h$ are continuous in $t$.\\

(H3.4) for some $L>0$ and $0<\alpha<1$ all $x,\,x'\in
\mathbb{R}^{n};\,y,\,y'\in\mathbb{R};\,z,\,z'\in\mathbb{R}^{d};\,v,\,v'\in
\mathcal {U}$
\begin{eqnarray*}
&&|f(t,x,y,z,v)-f(t,x',y',z',v')|+|\Phi(x)-\Phi(x')|\\
&\leq& L(|x-x'|+|y-y'|+|z-z'|+|v-v'|).
\end{eqnarray*}
\begin{eqnarray*}
\|g(t,x,y,z,v)-g(t,x',y',z',v')\|\leq
L(|x-x'|+|y-y'|)+\alpha|z-z'|).
\end{eqnarray*}

(H3.5) The function $g\in \mathbf{L}^2(R^d,\rho(x)dx).$\\

(H3.6) $\forall(y, z) \in R \times R^{d}, f(\cdot, y, z) \in \mathbf{H}^{2}, g(\cdot, y, z) \in \mathbf{H}^{2}$.\\

(H3.7) $f$ is measurable in $(t,x,y,z,v)$ and for any $r\in[t,T],$
$$E\int_0^T|f(r,0,0,0,v_r)|^2dr\leq M,$$
functions $f$ and $g$ are continuous and controlled by
$C(1+|x|+|y|+|z|+|v|)$.

  Then there exists a unique solution $(Y^{t,\zeta;v},Z^{t,\zeta;v})\in\mathcal
{S}^{p}(0,T;\mathbb{R}^{k})\times\mathcal
{H}^{p}(0,T;\mathbb{R}^{k\times d})$.
\vspace{1mm}Moreover, we get
the following estimates for the solution from Proposition 2.1 and
2.2.

\vspace{1mm} {\bf Proposition 3.3}\ \quad{\it

\begin{eqnarray}
\mathbb{E}\left\{\sup_{t\leq s\leq
T}|Y_{s}^{t,\zeta;v}|^{p}+(\int^T_t|Z_{r}^{t,\zeta;v}|^{2}dr)^{\frac{p}{2}}\right\}\leq
\mathbb{E}C_{p}(1+|\zeta|^{q}).
\end{eqnarray}}

 \vspace{1mm} {\bf Proposition 3.4}{\it

\begin{equation}
\begin{array}{lcl} &&\mathbb{E}\left\{\sup_{0\leq s\leq
T}|Y_{s}^{t,\zeta;v}-Y_{s}^{t',\zeta';v'}|^{p}+\left(\int^T_0|Z_{r}^{t,\zeta;v}-Z_{r}^{t',\zeta';v'}|^{2}dr
\right)^{\frac{p}{2}}\right\}\\
&\leq&
\mathbb{E}\left\{C_{p}(1+|\zeta|^{q}+|\zeta'|^{q})(|t-t'|^{\frac{p}{2}}+
|\zeta-\zeta'|^{p}+\int^T_0|v_{r}-v'_{r}|^{p}dr)\right\}.
\end{array}
\end{equation}

  The proof is complicated and technical, we put in the Appendix.}

  Given a control process $v(\cdot)\in\mathcal {U}$, we introduce the
associated cost functional:

\begin{eqnarray}
J(t,x;v(\cdot)):=Y_{s}^{t,x;v}|_{s=t},\qquad(t,x)\in[0,T]\times\mathbb{R}^{n},
\end{eqnarray}

  and we define the value function of the stochastic optimal control
problem

\begin{eqnarray}
u(t,x):=ess\sup_{v(\cdot)\in\mathcal
{U}}J(t,x;v(\cdot)),\qquad(t,x)\in[0,T]\times\mathbb{R}^{n},
\end{eqnarray}

  Now we continue our study of the control problem (12) and prove
that the celebrated dynamic programming principle still holds for
this optimization problem. Some proof ideas come from the proof of
the dynamic programming principle for recursive problem given by
Peng in Chinese version [6], and wu and Yu in [7].

  Now we introduce the following subspace of $\mathcal {U}$ :
\begin{equation*}
\begin{array}{ll}
\mathcal {U}^{t}:=\left\{v(\cdot)\in\mathcal
{U}\,|\,v(s)~is~\mathcal
{F}^{W}_{t,\,s}\vee\mathcal{F}^{B}_{s,\,T}~progressively~measurable,\,\forall
t\leq s\leq T\right\},\\
\overline{\mathcal
{U}}^{t}:=\left\{v_{s}=\sum^{N}_{j=1}v_{s}^{j}I_{A_{j}}|v_{s}^{j}\in\mathcal
{U}^{t},\{A_{j}\}_{j=1}^{N}~is~a~partition~of~(\Omega,\mathcal
{F}_{t}^{W})\right\}.
\end{array}
\end{equation*}
 Firstly we will prove that:

\vspace{1mm} {\bf Proposition 3.5}\ \quad{\it Under the assumptions
(H3.1)-(H3.4), the value function $u(t,x)$ defined in (12) is
$\mathcal {F}^{B}_{t,\,T}$ measurable.}
\begin{proof}
First we can prove:

$$
ess\sup_{v(\cdot)\in\mathcal
{U}}J(t,x;v(\cdot))=ess\sup_{v(\cdot)\in\overline{\mathcal
{U}}^{t}}J(t,x;v(\cdot)),
$$

$\overline{\mathcal {U}}^{t}$ is the subset of $\mathcal {U}$, then

$$
ess\sup_{v(\cdot)\in\mathcal {U}}J(t,x;v(\cdot))\geq
ess\sup_{v(\cdot)\in\overline{\mathcal {U}}^{t}}J(t,x;v(\cdot)).
$$

  We need to consider the inverse inequality. For any
$v(\cdot),\,\widetilde{v}(\cdot)\in\mathcal {U}$, for the
Proposition 3.4, we know

$$
\mathbb{E}\left\{|Y^{t,x;v}_{t}-Y^{t,x;\widetilde{v}}_{t}|^{p}\right\}\leq
C\mathbb{E}\int^T_t |v_{r}-\widetilde{v}_{r}|^{p}ds.
$$

  Note that $\overline{\mathcal {U}}^{t}$ is dense in $\mathcal
{U}$, then for each $v(\cdot)\in\mathcal {U}$, there exists a
sequence $\{v_{n}(\cdot)\}_{n=1}^{\infty}\in\overline{\mathcal
{U}}^{t}$ such that

$$
\lim_{n\rightarrow\infty}\mathbb{E}\left\{|Y^{t,x;v_{n}}_{t}-Y^{t,x;v}_{t}|^{p}\right\}=0.
$$

  So there exists a subsequence, we denote without loss of
generality $\{v_{n}(\cdot)\}_{n=1}^{\infty}$ such that

$$
\lim_{n\rightarrow\infty}Y^{t,x;v_{n}}_{t}=Y^{t,x;v}_{t},\qquad a.s.,
$$

so that

$$
\lim_{n\rightarrow\infty}J(t,x,v_{n}(\cdot))=J(t,x,v(\cdot)),\qquad
a.s..
$$

  By the arbitrariness of $v(\cdot)$ and the definition of essential
supremum, we get

$$
ess\sup_{v(\cdot)\in\overline{\mathcal {U}}^{t}}J(t,x;v(\cdot))\geq
ess\sup_{v(\cdot)\in\mathcal {U}}J(t,x;v(\cdot)),
$$

  then we obtain (3.11).

   Second, we want to prove

$$
ess\sup_{v(\cdot)\in\overline{\mathcal {U}}^{t}}J(t,x;v(\cdot))=
ess\sup_{v(\cdot)\in\mathcal {U}^{t}}J(t,x;v(\cdot)).
$$

  Obviously,

$$
ess\sup_{v(\cdot)\in\overline{\mathcal {U}}^{t}}J(t,x;v(\cdot))\geq
ess\sup_{v(\cdot)\in\mathcal {U}^{t}}J(t,x;v(\cdot)).
$$
\end{proof}
  In order to get the inverse inequality, we need the following
  Lemma:

 {\bf Lemma 3.6}
\begin{equation*}
\begin{array}{ll}
X^{t,x;\sum^{N}_{j=1}v^{j}I_{A_{j}}}=\sum^{N}_{j=1}I_{A_{j}}X^{t,x;v^{j}},\quad Y^{t,x;\sum^{N}_{j=1}v^{j}I_{A_{j}}}=\sum^{N}_{j=1}I_{A_{j}}Y^{t,x;v^{j}}\\
Z^{t,x;\sum^{N}_{j=1}v^{j}I_{A_{j}}}=\sum^{N}_{j=1}I_{A_{j}}Z^{t,x;v^{j}}
\end{array}
\end{equation*}
$\forall v(\cdot)\in\overline{\mathcal {U}}^{t}$, we have

$$
J(t,x;v(\cdot))=J(t,x;\sum^{N}_{j=1}v^{j}(\cdot)I_{A_{j}})=\sum^{N}_{j=1}I_{A_{j}}J(t,x;v^{j}(\cdot)),
$$

  because

\begin{eqnarray*}
&&ess\sup_{v(\cdot)\in\overline{\mathcal {U}}^{t}}J(t,x;v(\cdot))=
ess\sup_{v(\cdot)\in\overline{\mathcal
{U}}^{t}}\sum^{N}_{j=1}I_{A_{j}}J(t,x;v^{j}(\cdot))\\
&\leq&\sum^{N}_{j=1}ess\sup_{v(\cdot)\in\overline{\mathcal
{U}}^{t}}J(t,x;v^{j}(\cdot))=\sum^{N}_{j=1}ess\sup_{v(\cdot)\in\overline{\mathcal
{U}}^{t}}J(t,x;v(\cdot)=ess\sup_{v(\cdot)\in\mathcal
{U}^{t}}J(t,x;v(\cdot)),
\end{eqnarray*}
then we can get
$$
ess\sup_{v(\cdot)\in\overline{\mathcal {U}}^{t}}J(t,x;v(\cdot))\leq
ess\sup_{v(\cdot)\in\mathcal {U}^{t}}J(t,x;v(\cdot)).
$$
However, when $v(\cdot)\in\mathcal {U}^{t}$, the cost functional
 $J(t,x;v(\cdot))$ is $\mathcal {F}^{B}_{t,\,T}$ measurable.

  So

$$
u(t,x)=ess\sup_{v(\cdot)\in\mathcal {U}^{t}}J(t,x;v(\cdot))
$$

is $\mathcal {F}^{B}_{t,T}$ measurable.

Next we will discuss the continuity of value function $u(t,x)$ with
 respect to $x$ and $t$. We have the following estimates:

  {\bf Lemma 3.7}\quad {\it For each $t\in[0,T]$, $x$ and $x'\in\mathbb{R}^{n}$, we
have

\begin{equation}
\begin{array}{l l}
(i)\ \mathbb{E}|u(t,x)-u(t',x')|^{p}\leq
C_{p}(1+|x|^{q}+|x'|^{q})(|x-x'|^{p}+|t-t'|^{\frac{p}{2}});\\
(ii)\ \mathbb{E}|u(t,x)|^{p}\leq C_{p}(1+|x|^{q}).
\end{array}
\end{equation}
}

\begin{proof} Using the estimates: $\mathbb{E}(\sup_{t\leq s\leq
T}|Y_{s}^{t,x;v}|^{p})\leq C_{p}(1+|X|^{q})$, for each admissible
control $v(\cdot)\in\mathcal {U}$, we have

$$
\mathbb{E}|J(t,x;v(\cdot))|^{p}\leq C_{p}(1+|x|^{q})
$$
and
$$
\mathbb{E}|J(t,x;v(\cdot))-J(t',x;v'(\cdot))|^{p}\leq
C_{p}(1+|x|^{q}+|x'|^{q})(|t-t'|^{\frac{p}{2}}+|x-x'|^{p}).
$$

  On the otherhand, for each $\varepsilon>0,\,\exists
v(\cdot),\,v'(\cdot)\in\mathcal {U}$ such that:

\begin{eqnarray*}
J(t,x;v'(\cdot))\leq u(t,x)\leq J(t,x,v(\cdot))+\varepsilon,\hspace{3mm}\\
J(t',x';v(\cdot))\leq u(t',x')\leq J(t',x',v'(\cdot))+\varepsilon.
\end{eqnarray*}

\vspace{1mm}Form the estimate (10) we can get:

\begin{eqnarray*}
-C_{p}(1+|x|^{q})-\varepsilon\leq\mathbb{E}|J(t,x;v'(\cdot))|^{p}\leq
\mathbb{E}|u(t,x)|^{p}\leq
\mathbb{E}|J(t,x;v(\cdot))|^{p}+\varepsilon\leq
C_{p}(1+|x|^{q})+\varepsilon.
\end{eqnarray*}

  Form the arbitrariness of $\varepsilon$, we can obtain
(ii).

Similarly,
\begin{eqnarray*}
J(t,x;v'(\cdot))-J(t',x';v'(\cdot))-\varepsilon\leq
u(t,x)-u(t',x')\leq J(t,x;v(\cdot))-J(t',x';v(\cdot))+\varepsilon.
\end{eqnarray*}
\begin{eqnarray*}
|u(t,x)-u(t,x)\leq
\max\{|J(t,x;v(\cdot))-J(t',x';v'(\cdot))|,|J(t,x;v(\cdot))-J(t',x';v(\cdot))|\}+\varepsilon.
\end{eqnarray*}
\begin{eqnarray*}
&~~&\mathbb{E}|u(t,x)-u(t,x)|^{p}\\
&\leq&
C\max\{\mathbb{E}|J(t,x;v(\cdot))-J(t',x';v'(\cdot))|^{p},\mathbb{E}|J(t,x;v(\cdot))-J(t',x';v(\cdot))|^{p}\}+C\varepsilon^{p}\\
&\leq&
C_{p}(1+|x|^{q}+|x'|^{q})(|x-x'|^{p}+|t-t'|^{\frac{p}{2}})+C\varepsilon^{p}.
\end{eqnarray*}

  Then we can obtain (i).
\end{proof}
\vspace{2mm}For the value function of our recursive optimal control
problem.We have:

{\bf Lemma 3.8}\quad{\it$\forall t\in[0,T],\,\forall
v(\cdot)\in\mathcal {U}$, for all $\zeta\in L^{p}(\Omega,\mathcal
{F}_{t},P;)$, we have

$$
J(t,\zeta;v(\cdot))=Y^{t,\,\zeta;\,v(\cdot)}.
$$}

\vspace{1mm} \begin{proof}We first study a simple case: $\zeta$
is the following form: $\zeta=\sum^{N}_{i=1}I_{A_{i}}x_{i}$, where
$\{A_{i}\}^{N}_{i=1}$ is a finite partition of $(\Omega,\mathcal
{F}^{W}_{t})$, and $x_{i}\in\mathbb{R}^{n}$ for $1\leq i\leq N$, so

\begin{eqnarray*}
Y^{t,\,\zeta;\,v}_{s}=Y^{t,\sum^{N}_{i=1}I_{A_{i}}xi;v}_{s}=\sum^{N}_{i=1}I_{A_{i}}Y^{t,\,xi;\,v}_{s}.
\end{eqnarray*}

  From the definition of cost functional. We deduce that

\begin{eqnarray*}
Y^{t,\zeta;v}_{t}=\sum^{N}_{i=1}I_{A_{i}}Y^{t,xi;v}_{t}=\sum^{N}_{i=1}I_{A_{i}}J(t,x_{i};v(\cdot))
=J(t,\sum^{N}_{i=1}I_{A_{i}}x_{i};v(\cdot))=J(t,\zeta;v(\cdot)).
\end{eqnarray*}
  Therefor, for simple functions, we get the desired result.

\vspace{2mm}Given a general $\zeta\in L^{p}(\Omega,\mathcal
{F}_{t},P;\mathbb{R}^{n})$, we can choose a sequence of simple
function $\{\zeta_{i}\}$ which converges to $\zeta$ in
$L^{p}(\Omega,\mathcal {F}_{t},P;\mathbb{R}^{n})$. Consequently, we
have:
\begin{eqnarray*}
&&\mathbb{E}\{|Y^{t,\zeta;v}_{t}-Y^{t,\zeta_{i};v}_{t}|^{p}\}\\
&\leq&\mathbb{E}\{C_{p}(1+|\zeta|^{q}+|\zeta_{i}|^{q})(|\zeta-\zeta_{i}|^{p})\}\rightarrow
0,\quad as\,i\rightarrow \infty.
\end{eqnarray*}
So
\begin{eqnarray*}
&&\mathbb{E}\{|J(t,\zeta;v(\cdot))-J(t,\zeta_{i};v(\cdot))|^{p}\}\\
&\leq&\mathbb{E}\{C_{p}(1+|\zeta|^{q}+|\zeta_{i}|^{q})(|\zeta-\zeta_{i}|^{p})\}\rightarrow
0,\quad as\,i\rightarrow\infty.
\end{eqnarray*}
\vspace{2mm}With the help of $Y^{t,\zeta;v}_{t}=J(t,\zeta
;v(\cdot))$, the proof is completed.
\end{proof}
   For the value function of our recursive optimal
control problem, we have

 \vspace{2mm} {\bf Lemma 3.9:}\quad {\it Fixed $t\in[0,T)$ and $\zeta\in L^{p}(\Omega,\mathcal
{F}_{t},P;\mathbb{R}^{n})$, for each $v(\cdot)\in\mathcal {U}$, we
have :

$$
u(t,\zeta)\geq Y^{t,\,\zeta;\,v(\cdot)}_{t}.
$$

On the other hand, for each $\varepsilon>0$, there exists an
admissible control $v(\cdot)\in\mathcal {U}$ such that:

$$
u(t,\zeta)\leq Y^{t,\,\zeta,;\,v(\cdot)}_{t}+\varepsilon,\quad a.s..
$$}

\vspace{2mm}Now we start to discuss the (generalized) dynamic
programming principle for our recursive optimal control problem.

   Firstly we introduce a family of (backward) semigroups which is
original from Peng's idea in [6].

   Given the initial condition $(t,x)$, an admissible control
$v(\cdot)\in\mathcal {U}$, a positive number $\delta\leq T-t$ and a
real-value random variable $\eta\in L^{p}(\Omega,\mathcal
{F}_{t+\delta},P;\mathbb{R})$, we denote

$$
G^{t,x;v}_{t,t+\delta}[\eta]:=Y_{t}\;,
$$

  where $(Y_{s},Z_{s})$ is the solution of the following double BSDE
with the horizon $t+\delta$:

$$
Y_{s}=\xi+\int^{t+\delta}_s
f(r,Y_{r},Z_{r})dr+\int^{t+\delta}_sg(r,Y_{r},Z_{r})d{B_{r}}
-\int^{t+\delta}_sZ_{r} d W_{r},\ t\leq s\leq t+\delta.
$$

  Obviously,

$$
G^{t,x;v}_{t,T}[\Phi(X^{t,x;v}_{T})]=G^{t,x;v}_{t,t+\delta}[Y^{t,x;v}_{t+\delta}].
$$

   Then our (generalized) dynamic programming principle holds.

 \vspace{2mm}{\bf Theorem 3.11} \quad {\it Under the assumption
(H3.1)-(H3.4), the value function $u(t,x)$ obeys the following
dynamic programming principle: for each $0<\delta\leq T-t$,

$$
u(t,x)=ess\sup_{v(\cdot)\in\mathcal
{U}}G^{t,x;v}_{t,t+\delta}[u(t+\delta,X^{t,x;v}_{t+\delta})].
$$}

\begin{proof}We have

\begin{eqnarray*}
u(t,x)=ess\sup_{v(\cdot)\in\mathcal
{U}}G^{t,x;v}_{t,T}[\Phi(X^{t,x;v}_{T})]=ess\sup_{v(\cdot)\in\mathcal
{U}}G^{t,x;v}_{t,t+\delta}[Y^{t,x;v}_{t+\delta}]=ess\sup_{v(\cdot)\in\mathcal
{U}}G^{t,x;v}_{t,t+\delta}[Y^{t+\delta,X^{t,x;v}_{t+\delta};v}_{t+\delta}].
\end{eqnarray*}

  Form Lemma 3.10 and the comparison theorem of double BDSDE

\begin{eqnarray*}
u(t,x)\leq ess\sup_{v(\cdot)\in\mathcal
{U}}G^{t,x;v}_{t,t+\delta}[u(t+\delta,X^{t,x;v}_{t+\delta})].
\end{eqnarray*}

   On the other hand, from Lemma 3.10, for every $\varepsilon>0$, we can
find an admissible control $\overline{v}(\cdot)\in\mathcal {U}$ such
that

$$
u(t+\delta,X^{t,x;v}_{t+\delta})\leq
Y^{t+\delta,X^{t,x;v}_{t+\delta};\overline{v}}_{t+\delta}+\varepsilon.
$$

  For each $v(\cdot)\in\mathcal {U}$, we denote
$\widetilde{v}(s)=I_{\{s\leq
t+\delta\}}v(s)+I_{\{s>t+\delta\}}\overline{v}(s)$. From the above
inequality and the comparison theorem, we get

$$
Y^{t+\delta,X^{t,x;\widetilde{v}}_{t+\delta};\widetilde{v}}_{t+\delta}\geq
u(t+\delta,X^{t,x;\widetilde{v}}_{t+\delta})-\varepsilon,\quad
u(t,x)\geq ess\sup_{\widetilde{v}(\cdot)\in\mathcal
{U}}G^{t,x;\widetilde{v}}_{t,t+\delta}[u(t+\delta,X^{t,x;\widetilde{v}}_{t+\delta})-\varepsilon].
$$

  By Proposition 2.2 , there exists a positive constant $C_{0}$ such that

\begin{eqnarray*}
u(t,x)\geq ess\sup_{\widetilde{v}\in\mathcal
{U}}G^{t,x;\widetilde{v}}_{t,t+\delta}[u(t+\delta,X^{t,x;\widetilde{v}}_{t+\delta})]-C_{0}\varepsilon.
\end{eqnarray*}

  Therefore, letting $\varepsilon\downarrow0$, we obtain

$$
u(t,x)\geq ess\sup_{\widetilde{v}\in\mathcal
{U}}G^{t,x;\widetilde{v}}_{t,t+\delta}[u(t+\delta,X^{t,x;\widetilde
{v}}_{t+\delta})].
$$

  Because $\widetilde{v}(\cdot)$ acts only on $[t,t+\delta]$ for
$G^{t,x;\widetilde{v}}_{t,t+\delta}$, from the definition of
$\widetilde{v}(\cdot)$ and the arbitrariness of
$\widetilde{v}(\cdot)$, we know that the above inequality can be
written as

$$
u(t,x)\geq ess\sup_{v\in\mathcal
{U}}G^{t,x;v}_{t,t+\delta}[u(t+\delta,X^{t,x;v}_{t+\delta})],
$$
which is our desired conclusion.
\end{proof}
\section{Sobolev weak solutions for the HJB equations corresponding to the stochastic recursive control problem }

In this section we consider the Sobolev weak solution for the SHJB
equation related to the stochastic recursive optimal control problem.

We give some preliminary results of the BDSDE which are useful for the
sobolev weak solutions for the recursive optimal control problem. In order to facilitate understanding and narration, we divided it into several parts.\\
\textbf{\emph{ Part }}\uppercase\expandafter{\romannumeral1}\\
Consider the control system defined by (4)
\begin{eqnarray}\label{eq39}
\left\{\begin{array}{ll}
dx_{s}^{t,x;v}=b(s,x_{s}^{t,x;v},v_s)ds+\sigma(s,x_{s}^{t,x;v},v_s)dW_{s},\quad
s\in[t,
T],\\
x_{t}^{t,x;v}=x.
\end{array}\right.
\end{eqnarray}
satisfying the following conditions:\\
(H4.1) The coefficient $b$ is $2$ times continuously differentiable
in $x$ and all their partial derivatives are uniformly bounded,
$\sigma$ is $3$ times continuously differentiable in $x$ and all
their partial derivatives are uniformly bounded, and
$|b(t,x,v)|+|\sigma(t,x,v)|\leq K(1+|x|)$, where $K$ is a constant.\\
And the cost function defined by the following BSDE:
\begin{eqnarray*}\label{eq2.74}
Y^{t,\zeta;v}_{s}&=&h(x^{t,\zeta;v}_{T})+\int^T_s
f(r,x^{t,\zeta;v}_{r},Y^{t,\zeta;v}_{r},Z^{t,\zeta;v}_{r},v_r)dr+\int^T_s
g(r,x^{t,\zeta;v}_{r},Y^{t,\zeta;v}_{r},Z^{t,\zeta;v}_{r})dBr\\
&-&\int^T_s Z^{t,\zeta;v}_{r}dW_{r},
\end{eqnarray*}
where
\begin{eqnarray*}
&&h:R^n\rightarrow R,\hspace{44mm}\\
&&f:[0,T]\times R^n\times R\times R^d\times U\rightarrow R,\hspace{3mm}\\
&&g:[0,T]\times R^n\times R\times R^d\rightarrow R,
\end{eqnarray*}
satisfying the conditions as same as that denoted in Chapter 3.

Obviously, under the above assumptions(H3.4)(H3.5)(H3.7)and(H4,1), for a
given  control $v(\cdot)\in\mathcal {U}$, there exists a unique
solution $(Y^{t,\zeta;v},Z^{t,\zeta;v})\in S^{2}(0,T;R)\times
H^{2}(0,T;R^d)$. We introduce the associated cost functional:

\begin{eqnarray}\label{eq2.75}
J(t,x;v):=Y_{s}^{t,x;v}|_{s=t},\qquad(t,x)\in[0,T]\times R^n,
\end{eqnarray}
and  define the value function of the stochastic optimal control problem
\begin{eqnarray}\label{eq2.76}
u(t,x):=ess\sup_{v\in\mathcal {U}}J(t,x;v),\qquad(t,x)\in[0,T]\times
R^n,
\end{eqnarray}

According to the conclusion in previous chapter, we know that the celebrated dynamic programming principle still holds for this recursive stochastic optimal control problem. We
therefore deduce the following HJB equation:

\begin{eqnarray}\label{eq2.77}
\left\{\begin{array}{ll} \frac{\partial u}{\partial
t}(t,x)+\sup_{v\in\mathcal {U}}\{\mathcal
{L}(t,x,v)u(t,x)+f(t,x,u(t,x),\sigma\nabla u(t,x),v)+g(t,x,u(t,x),\sigma\nabla u(t,x))dBt\}=0,\\
u(T,x)=h(x),
\end{array}\right.
\end{eqnarray}
where $\mathcal {L}$ is a family of second order linear partial
differential operators,
\begin{eqnarray}\mathcal {L}(t,x,q)\varphi&=&\frac{1}{2}tr[\sigma(t,x,v_t)\sigma(t,x,v_t)^TD^2\varphi]\nonumber\\
&&+\langle b(t,x,v_t),D\varphi\rangle.\nonumber
\end{eqnarray}
\textbf{\emph{ Part }}\uppercase\expandafter{\romannumeral2}\\

We define the weight function $\rho$ is continuous positive on $R^d$ satisfying $\int_{R^d}\rho(x)dx=1$ and $\int_{R^d}|x|^2\rho(x)dx<\infty$.\\

Denote by $L^2(R^d,\rho(x)dx)$ the weighted $L^2$-space with weight
function endowed with the norm
$$\|u\|_{L^2(R^d,\rho(x)dx)}=[\int_{R^d}|u(x)|^2\rho(x)dx]^{\frac{1}{2}}.$$

 We set $D:=\{u:R^d\rightarrow R\ such\ that\ u\in
L^2(R^d,\rho (x)dx)\ and\ \frac{\partial u}{\partial x_i}\in
L^2(R^d,\rho (x)dx)\}$, where $\frac{\partial u}{\partial x_i}$
is derivative with respect to $x$ in the weak sense. Note that $D$ equipped with the norm
$$\|u\|_D=[\int_{R^d}|u(x)|^2\rho(x)dx+\sum_{1\leq i\leq d}\int_{R^d}\left|\frac{\partial u}{\partial x_i}\right|^2\rho(x)dx]^{\frac{1}{2}}$$
is a Hilbert space, which is a classical
Dirichlet space. Moreover, $D$ is a subset of the Sobolev space
$H_1(R^d)$.

 We set $H:=\{u: u\in L^2(R^d,\rho(x)dx)\ and \
(\sigma^\ast\bigtriangledown u)\in L^2(R^d, \rho(x)dx)\}$ equipped
with the norm
$$\|u\|_H=[\int_{R^d}|u(x)|^2\rho(x)dx+\int_{R^d}\left|(\sigma^\ast\nabla
u(x))\right|^2\rho(x)dx]^{\frac{1}{2}}.$$

 We say $u\in L^2([0,T],H)$ if
$\int^T_0\|u(t)\|^2_Hdt<\infty$.

Let $T$ be a strictly positive real number and $U$ a nonempty
compact set of $R^k$.\\
\textbf{\emph{ Part }}\uppercase\expandafter{\romannumeral3}

Then, we introduce some equivalence norm.

The solution of SDE generates a stochastic flow, and the inverse flow is
denoted by $\widehat{x}^{t,x,v}_s$. It is known from \cite{IW1989}
that $x\rightarrow\widehat{x}^{t,x,v}_s$ is differentiable and we
denote by $J(\widehat{x}^{t,x,v}_s)$ the determinant of the Jacobian
matrix of $\widehat{x}^{t,x,v}_s$, which is positive and
$J(\widehat{x}^{t,x,v}_t)=1$. For $\varphi \in C^\infty_c(R^d)$ we
define a process $\varphi_t: \Omega\times [0,T]\times R^d\rightarrow
R$ by
$\varphi_t(s,x)=\varphi(\widehat{x}^{t,x,v}_s)J(\widehat{x}^{t,x,v}_s)$.
Following Kunita \cite{K1994}, we can define the composition of
$u\in L^2(R^d)$ with the stochastic flow by $(u\circ
x^{t,\cdot,v}_s,\varphi)=(u,\varphi_t(s,\cdot))$. Indeed, by a
change of variable, we have
$$(u\circ x^{t,\cdot,v}_s,\varphi)=\int_{R^d}u(y)\varphi(\widehat{x}^{t,x,v}_s)J(\widehat{x}^{t,x,v}_s)dy=
\int_{R^d}u(x^{t,x,v}_s)\varphi(x)dx.$$

In \cite{s1}, V. Bally and A. Matoussi proved that
$\varphi_t(s,x)$ is a semimartingale and admits the following lemma 4.1 and lemma 4.2.

\begin{lemma}\label{lemma6} For $\varphi\in C^2_c(R^d)$, we have
$$\varphi_t(s,x)=\varphi(x)-\sum^d_{j=1}\int^s_t\sum^d_{i=1}\frac{\partial}{\partial
x_i}(\sigma_{i,j}(r,x)\varphi_t(r,x))dW^j_r+\int^s_tL^\ast_r\varphi_t(r,x)dr,$$
\end{lemma}
where $L^\ast_t$ is the adjoint operator of $L_t$.

The next lemma, known as the norm equivalence result and proved in
\cite{s1} plays an important role in the proof of the main
result.

\begin{lemma}\label{lemma7} Assume that
(H1) holds. Then for any $v\in \cal{U}$ there exist two
constants $c>0$ and $C>0$ such that for every $t\leq s\leq T$ and
$\varphi\in L^1(R^d;\rho(x)dx)$
\begin{eqnarray}
c\int_{R^d}|\varphi(x)|\rho(x)dx\leq\int_{R^d}E(|\varphi(x^{t,x;v}_s)|)\rho(x)dx\leq
C\int_{R^d}|\varphi(x)|\rho(x)dx.\nonumber
\end{eqnarray}
Moreover, for every $\psi\in L^1([0,T]\times R^d;dt\otimes
\rho(x)dx)$,
\begin{eqnarray}
c\int_{R^d}\int^T_t|\psi(s,x)|ds\rho(x)dx\leq\int_{R^d}\int^T_tE(|\psi(s,x^{t,x;v}_s)|)ds\rho(x)dx\leq
C\int_{R^d}\int^T_t|\psi(s,x)|ds\rho(x)dx.\nonumber
\end{eqnarray}
\end{lemma}

The constants $c$ and $C$ depend on $T$, on $\rho$ and on the bounds
of derivatives of the $b$ and $\sigma$. The proof is similar to the
proof of Proposition 5.1 in \cite{s1}, hence we omit
it.\vspace{2mm}

Now we give the definition of a sobolev solution for SHJB equation(\ref{eq2.77}).\\
{\bf Definition 4.1}\quad We say that $V$ is a weak solution of the
equation (\ref{eq2.77}), if

{\rm(i)} $V\in L^2([0,T];H), i.e.,$
$$\int^T_0\|V(t)\|^2_Hdt=\int^T_0(\int_{R^d}|V(t,x)|^2\rho(x)dx+\int_{R^d}|(\sigma^\ast\nabla
V)(t,x)|^2\rho(x)dx)dt<\infty.
$$

{\rm(ii)} For any nonnegative $\varphi\in C^{1,\infty}_c([0,T]\times
R^d)$ and for any $v\in\mathcal {U}$,
\begin{eqnarray}\label{eq2.79}
&&\int_{R^d}\int^T_s(V(r,x),\partial_r\varphi(r,x))drdx+\int_{R^d}(V(s,x),\varphi(s,x))dx\nonumber\\
&\geq&\int_{R^d}(h(x),\varphi(T,x))dx+\int_{R^d}\int^T_s(f(r,x,V,\sigma^\ast
\nabla V,v_r),\varphi(r,x))drdx\nonumber\\
&+&\int_{R^d}\int^T_s(g(r,x,V,\sigma^\ast
\nabla V),\varphi(r,x))dBrdx+\int_{R^d}\int^T_s(\mathcal {L}^v_rV(r,x),\varphi(r,x))drdx,
\end{eqnarray}
where $(L_rV(r,x),\varphi(r,x))=\int_{R^d}(\frac{1}{2}(\nabla
V\sigma)(\sigma^\ast\nabla\varphi)+Vdiv(b-A)\varphi )dx$ with
$A_i=\frac{1}{2}\sum^d_{k=1}\frac{\partial a_{k,i}}{\partial x_k}$.

{\rm(iii)} For any nonnegative $\varphi\in
C^{1,\infty}_c([0,T]\times R^d)$ and for any small $\varepsilon>0$,
there exists a control $v'\in\mathcal {U}$, such that
\begin{eqnarray}\label{eq2.80}
&&\int_{R^d}\int^T_s(V(r,x),\partial_r\varphi(r,x))drdx+\int_{R^d}(V(s,x),\varphi(s,x))dx-\varepsilon\nonumber\\
&\leq&\int_{R^d}(h(x),\varphi(T,x))dx+\int_{R^d}\int^T_s(f(r,x,V,\sigma^\ast
\nabla V,v'_r),\varphi(r,x))drdx\nonumber\\
&+&\int_{R^d}\int^T_s(g(r,x,V,\sigma^\ast
\nabla V),\varphi(r,x))dBrdx+\int_{R^d}\int^T_s(\mathcal {L}^{v'}_rV(r,x),\varphi(r,x))drdx.
\end{eqnarray}

\begin{lemma} Let $(\xi, f, g) $ and $(\xi^{\prime}, f^{\prime}, g) $ be two parameters of BDSDEs, each one satisfies all  the assumptions (H1), (H2) and (H3) with the exception that the Lipschitz condition  could be  satisfied by either $f $ or $f^{\prime} $ only and suppose in addition the following
\begin{eqnarray}
\xi \leq \xi^{\prime}, a . s, \quad f(t, y, z) \leq f^{\prime}(t, y, z), \text { a.s. a.e. } \quad \forall(y, z) \in R \times R^{d}.
\end{eqnarray}
Let $(Y, Z) $ be a solution of the BDSDE with parameter$(\xi, f, g) $ and $(Y^{\prime}, Z^{\prime})$  a solution of the  BDSDE with parameter $(\xi^{\prime}, f^{\prime}, q) $.  Then

\begin{eqnarray}
Y_{t} \leq Y_{t}^{\prime}, \quad \text { a.e. }  \quad \forall ~0\leq t \leq T.
\end{eqnarray}
\end{lemma}
\noindent
The proof is similar to the proof in \cite{SGK2005}.

\begin{lemma}\label{lemma3}
Let(H3.4)(H3.5)(H3.7)and(H4,1) hold, then for any $v\in \mathcal{U}$, the value
function satisfies
\begin{equation*}
\begin{aligned}\label{eq25}
&V(s,x^{t,x,v}_s)\geq\\
&E\{\int^{s'}_sf(r,x^{t,x,v},y^{t,x,v}_r,z^{t,x,v},v_r)dr+g(r,x^{t,x,v},y^{t,x,v}_r,z^{t,x,v})dBr+&V(s',x^{t,x,v}_{s'})|\mathcal
{F}^{W}_{t,\,s}\vee\mathcal{F}^{B}_{s,\,T}\}\\
\end{aligned}
\end{equation*}
\begin{flalign}
& &\forall t\leq s\leq s'\leq T.
\end{flalign}
 and for any small $\varepsilon>0$, there exists a $v'\in \mathcal{U}$,
such that
\begin{equation*}
\begin{aligned}\label{eq26}
&V(s,x^{t,x,v'}_s)-\varepsilon\leq\\
&E\{\int^{s'}_sf(r,x^{t,x,v'},y^{t,x,v'}_r,z^{t,x,v'},v_r)dr+g(r,x^{t,x,v'},y^{t,x,v'}_r,z^{t,x,v'})&dBr+V(s',x^{t,x,v'}_{s'})|\mathcal
{F}^{W}_{t,\,s}\vee\mathcal{F}^{B}_{s,\,T}\}\\
\end{aligned}\end{equation*}
\begin{flalign}
& &\forall t\leq s\leq s'\leq T.
\end{flalign}
\end{lemma}
\begin{proof}According to the theory of dynamic programming principle we have got above,
\begin{equation}
V\left(s, x_{s}^{t , x, v}\right)=\operatorname{ess} \sup _{v \in \mathcal{U}} G_{s, s^{\prime}}^{t, x, v}\left[V\left(s^{\prime}, x_{s^{\prime}}^{t, x, v}\right)\right], \quad \forall t \leq s \leq s^{\prime} \leq T.
\end{equation}
Then we set
\begin{equation}
G_{s, s^{\prime}}^{t, x, v}\left[V\left(s^{\prime}, x_{s^{\prime}}^{t, x, v}\right)\right]:=\widetilde{y}_{s}^{t, x, v}
\end{equation}
is the solution of following BDSDE:
\begin{eqnarray}
\widetilde{y}_{s}^{t, x, v}&=&V(s^{\prime}, x_{s^{\prime}}^{t, x, v})+\int_{s}^{s^{\prime}} f(r, x_{r}^{t, x, v}, \widetilde{y}_{r}^{t, x, v},\widetilde{z}_{r}^{t, x, v},v_{r}) d r+\int_{s}^{s^{\prime}} g(r, x_{r}^{t, x, v}, \widetilde{y}_{r}^{t, x, v},\widetilde{z}_{r}^{t, x, v}) dB r\nonumber\\
&-&\int_{s}^{s^{\prime}} \widetilde{z}_{r}^{t, x, v} d W_{r}, ~~~i.e.,
\end{eqnarray}

\begin{equation}
\widetilde{y}_{s}^{t, x, v}=E\left\{\int_{s}^{s^{\prime}} f(r, x_{r}^{t, x, v}, \widetilde{y}_{r}^{t, x, v},\widetilde{z}_{r}^{t, x, v},v_{r}) d r+g(r,x^{t,x,v},\widetilde y^{t,x,v}_r,\widetilde{z}_{r}^{t, x, v}dBr+V\left(s^{\prime}, x_{s^{\prime}}^{t, x, v}\right) | \mathcal
{F}^{W}_{t,\,s}\vee\mathcal{F}^{B}_{s,\,T}\right\}.
\end{equation}
Then it is no hard to finish the proof.
\end{proof}
\begin{lemma} For each $t\in[0,T]$, $x$ and $x'\in\mathbb{R}^{n}$, we have

{\rm(i)} $\left(Y_{t}^{n}-V_{t}\right)^{-}\rightarrow0 $ in $S^{2};$

{\rm(ii)} $\left(Y_{t}^{p}-V_{t}\right)^{-}\rightarrow0 $ in $S^{2}$.
\end{lemma}
\begin{proof}
The proof is similar to the proof in  \cite{AM2007}. Since $Y_{t}^{n} \geq Y_{t}^{0}$,  we can  replace $V_{t}$ by $V_{t} \vee Y_{t}^{0} ,$ so assume that $E\left(\sup _{t \leq T} V_{t}^{2}\right)<\infty$ ,We first want to compare a.s. $Y_{t}$ and $S_{t}$ for all $t \in[0, T],$ while we do not know yet that $Y$ is a.s. continuous. From the comparison theorem for BDSDE's, we have that a.s. $Y_{t}^{n} \geq \tilde{Y}_{t}^{n}, 0 \leq t \leq T$ $n \in N,$ where $\left\{\tilde{Y}_{t}^{n}, \tilde{Z}_{t}^{n} ; 0 \leq t \leq T\right\}$ is the unique solution of the BDSDE:
 \begin{eqnarray}
\tilde{Y}_{t}^{n}&=&\xi+\int_{t}^{T} f(s, X_s,Y_{s}^{n}, Z_{s}^{n},V_s) d s+n \int_{t}^{T}(V_{t}-\tilde{Y}_{s}^{n}) d s+\int_{t}^{T} g(s, X_s,Y_{s}^{n}, Z_{s}^{n}) d B_{s}\nonumber\\
&-&\int_{t}^{T} \tilde{Z}_{s}^{n} d W_{s}.
 \end{eqnarray}
Let $\nu$ be a stopping time such that $0 \leq \nu \leq T .$ Then
\begin{eqnarray}
\tilde{Y}_{t}^{n}&=& E^{\mathcal{F}_{\nu}}[e^{-n(T-\nu)} \xi+\int_{\nu}^{T} e^{-n(s-\nu)} f(s, X_s,Y_{s}^{n}, Z_{s}^{n},V_s) d s+n \int_{\nu}^{T} e^{-n(s-\nu)} V_{s} d s] \nonumber\\ &+&\int_{\nu}^{T} e^{-n(s-\nu)} g(s,X_s Y_{s}^{n}, Z_{s}^{n}) d B_{s}.
 \end{eqnarray}
It is easily seen that
$$
e^{-n(T-\nu)} \xi+n \int_{\nu}^{T} e^{-n(s-\nu)} V_{s} d s \rightarrow \xi \mathbf{1}_{\nu=T}+V_{\nu} \mathbf{1}_{\nu<T}
,$$
a.s. and in ${L}^{2},$ and the conditional expectation converges also in ${L}^{2} $. Moreover,
$$\left|\int_{\nu}^{T} e^{-n(s-\nu)} f\left(s, Y_{s}^{n}, Z_{s}^{n}\right) d s\right| \leq \frac{1}{\sqrt{2 n}}\left(\int_{0}^{T} f^{2}\left(s, Y_{s}^{n}, Z_{s}^{n}\right) d s\right)^{\frac{1}{2}},$$
hence $E^{\mathcal{F}_{\nu}} \int_{\nu}^{T} e^{-n(s-\nu)} f\left(s, Y_{s}^{n}, Z_{s}^{n}\right) d s \rightarrow 0$ in ${L}^{2},$ as $n \rightarrow \infty$
and\\
 \begin{eqnarray*}
E(\int_{\nu}^{T} g(s, Y_{s}^{n}, Z_{s}^{n}) d B_{s})^{2} &\leq &c E \int_{0}^{T} e^{-2 n(s-\nu)} g^{2}(s, Y_{s}^{n}, Z_{s}^{n}) d s\\
&\leq& \frac{c}{4 n} E \int_{0}^{T} g^{4}\left(s, Y_{s}^{n}, Z_{s}^{n}\right) d s \rightarrow 0.
 \end{eqnarray*}
Consequently, $\tilde{Y}_{s}^{n} \rightarrow \xi \mathbf{1}_{\nu=T}+S_{\nu} \mathbf{1}_{\nu<T}$ in mean square, and $Y_{\nu }\geq V_{\nu}$ a.s. From this and the section theorem in Dellacherie and Meyer \cite{DM1978}, it follows that a.s.
$$
Y_{t}^{n} \geq V_{t}, \quad 0 \leq t \leq T.
$$
Hence $(Y_{t}^{n}-V_{t})^{-} \searrow 0, 0 \leq t \leq T$, and from Dini's theorem the convergence is uniform in $t$.\\
Since $\left(Y_{t}^{n}-V_{t}\right)^{-} \leq\left(V_{t} -Y_{t}^{0}\right)^{+} \leq\left|V_{t}\right|+\left|Y_{t}^{0}\right|$, we have
$$
\lim_{n\rightarrow+\infty}E(\sup_{0\leq t\leq T}|Y_{t}^{n}-V_{t})^{-}|^{2})=0.
$$
by the dominated convergence theorem.
\end{proof}
\noindent
Before lemma 4.6, we now introduce the BDSDE with increasing process:
 \begin{eqnarray}
 Y_{t}=\xi+\int_{t}^{T} f\left(s, Y_{s}, Z_{s}\right) d s+K_{T}-K_{t}+\int_{t}^{T} g\left(s, Y_{s}, Z_{s}\right) d B_{s}-\int_{t}^{T} Z_{s} d W_{s}, 0 \leq t \leq T.
 \end{eqnarray}
The solution of the equation is triple $(Y,Z,K)$  of $\mathcal{F}_{t}$ measurable and take value in $( {R}, {R}^{\mathrm{d}} , {R}_{+})$ and satisfying \\
(H4.2) $Z \in \mathcal {H}^{2}$.\\
(H4.3) $Y \in \mathcal {S}^{2}, \text { and } K_{T} \in \mathcal {L}^{2}$.\\
(H4.4) ${K_t}$
is a continuous and increasing process, $K_{0}=0$ and $\int_{0}^{T}\left(Y_{t}-V_{t}\right) d K_{t}=0 $.\\
\begin{lemma}\label{lemma8} We assume (H3.4)(H3.5)(H3.7)(H4.1)-(H4.4), then $V(s, x^{t, x, v}_s)$ is a
g-super
solution and $E|V(s, x^{t, x, v}_s)|^2<\infty$. Moreover there
exists a unique increasing process $(K^{t, x,v}_r)$ with $K^{t, x, v}_t=0$ and
$E[(K^{t,x, v}_T)^2]<\infty$ such that $V(s, x^{t, x, v}_s)$ coincides with the
unique solution $y^{t, x, v}_s$ of the BSDE:
\begin{eqnarray}\label{eq2.82}
y^{t,x,v}_t&=&V(T,x^{t,x,v}_T)+\int^T_tf(r,x^{t,x,v}_r,y^{t,x,v}_r,Z^{t,x,v}_r,v_r)dr+K^{t,x,v}_T-K^{t,x,v}_t\nonumber\\
&+&\int^T_tg(r,x^{t,x,v}_r,y^{t,x,v}_r,Z^{t,x,v}_r)dBr-\int^T_tZ^{t,x,v}_rdW_r.
\end{eqnarray}
where $Z^{t,x,v'}_r=\sigma^\ast\nabla V(r,x^{t,x,v'}_r)$ in the
sense of Definition 4.1..
\end{lemma}
\begin{proof}Since the solution of the BDSDE is no longer a super-martingale, the method of proof in Lemma 4.1.\cite{WWZ}  will fail in our situation. The ides of proof comes from the the properties of BDSDE and limitation theory.
According to the penalization method and the comparasion theorem
\begin{eqnarray}
f_n(s,x,y,z,v)=f(s,x,y,z,v_s)+n(y-V_s)^+.
\end{eqnarray}
For each $n\in\mathcal {N}$, we ~denote $(Y^n,Z^n)$ the unique pair of $\mathcal{F}_t$ measureable process with valued in $R \times R^{d}$ is the solution of
 \begin{eqnarray}\label{eq2.85}
Y_{t}^{n}&=&V(T,X_T)+\int_{t}^{T} f\left(s, X,Y_{s}^{n}, Z_{s}^{n},V_s\right) d s+n \int_{t}^{T}\left(Y_{s}^{n}-V_{s}\right)^{-} d s\nonumber\\
&+&\int_{t}^{T} g\left(s, X,Y_{s}^{n}, Z_{s}^{n}\right) d B_{s}-\int_{t}^{T} Z_{s}^{n} d W_{s}.
\end{eqnarray}

We denote
$$K^{n}_t=\int _{0}^{t}(Y_{s}^{n}-V_{s})^{-}d {s}.$$\\
First we prove $(Y,Z)$ is the limit of $(Y^n,Z^n)$. We know $f_{n}(t, y, z) \leq f_{n+1}(t, y, z)$, from comparison theorem, $Y_t^{n}\leq Y_t^{n+1}, 0\leq t\leq T$. Therefore
 \begin{eqnarray}
 Y_{t}^{n} \uparrow Y_{t}, \quad 0 \leq t \leq T, \quad \text { a.e. }
\end{eqnarray}
Moreover $ y^{t,x,v}_t$ is bounded by $V(t,x^{t,x,v}_t)$ and
according to the result from \cite{AM2007}
 \begin{eqnarray}
E\left(\sup _{0 \leq t \leq T}\left|Y_{t}^{n}\right|^{2}\right)+E \int_{t}^{T}\left|Z_{s}^{n}\right|^{2} d s+E\left[\left(K_{T}^{n}\right)^{2}\right] \leq c, \quad n \in \mathbf{N}.
\end{eqnarray}
It follows from the Fatou lemma that
$E\left(\sup _{0 \leq t \leq T} |Y_{t}|^{2}\right) \leq c$,
then by the dominated convergence,
 \begin{eqnarray}\label{eq2.83}
E \int_{0}^{T}\left(Y_{t}-Y_{t}^{n}\right)^{2} d t \rightarrow 0, \quad \text { as } \quad n \rightarrow \infty.
 \end{eqnarray}
 Next, we desire to prove $
E \int_{0}^{T}\left(Z_{t}-Z_{t}^{n}\right)^{2} d t \rightarrow 0, $ as $ \quad n \rightarrow \infty$.
Applying It\^{o}'s formula to the proces$|Y_{t}^n-Y_{t}^p|^{2}$.

\begin{eqnarray*}
&\hspace{2mm}&\left|Y_{t}^{n}-Y_{t}^{p}\right|^{2}+\int_{t}^{T}\left|Z_{s}^{n}-Z_{s}^{p}\right|^{2} d s \\
&=&2 \int_{t}^{T}\left[f\left(s, X_s, Y_{s}^{n}, Z_{s}^{n},V_s\right)-f\left(s, X_s, Y_{s}^{p}, Z_{s}^{p},V_s\right)\right]\left(Y_{s}^{n}-Y_{s}^{p}\right) d s\\
&+&\int_{t}^{T}\left|g\left(s, X_s, Y_{s}^{n}, Z_{s}^{n}\right)-g\left(s, X_s,Y_{s}^{p}, Z_{s}^{p}\right)\right|^{2} d s \\
&+&2 \int_{t}^{T}\left[g\left(s, X_s, Y_{s}^{n}, Z_{s}^{n}\right)-g\left(s,  X_s,Y_{s}^{p}, Z_{s}^{p}\right)\right]\left(Y_{s}^{n}-Y_{s}^{p}\right) d B_{s}-2 \int_{t}^{T}\left(Y_{s}^{n}-Y_{s}^{p}\right)\left(Z_{s}^{n}-Z_{s}^{p}\right) d W_{s}\\
&+&\int_{t}^{T}\left(Y_{s}^{n}-Y_{s}^{p}\right) d\left(K_{s}^{n}-K_{s}^{p}\right).
\end{eqnarray*}

\begin{eqnarray*}
&\hspace{2mm}&E\left(\left|Y_{t}^{n}-Y_{t}^{p}\right|^{2}\right)+E \int_{t}^{T}\left|Z_{s}^{n}-Z_{s}^{p}\right|^{2} d s\\
&\leq&2KE \int_{t}^{T}\left(\left|Y_{s}^{n}-Y_{s}^{p}\right|^{2}+\left|Y_{s}^{n}-Y_{s}^{p}\right| \cdot\left|Z_{s}^{n}-Z_{s}^{p}\right|\right) d s+K E \int_{t}^{T}\left|Y_{s}^{n}-Y_{s}^{p}\right|^{2} d s\\
&+&\alpha E \int_{t}^{T}\left|Z_{s}^{n}-Z_{s}^{p}\right|^{2} d s+2 E \int_{t}^{T}\left(Y_{s}^{n}-V_{s}\right)^{-} d K_{s}^{p}+2 E \int_{t}^{T}\left(Y_{s}^{p}-V_{s}\right)^{-} d K_{s}^{n}\\
&\leq&\left(3 K+K^{2} \frac{2}{1-\alpha}\right) E \int_{t}^{T}\left|Y_{s}^{n}-Y_{s}^{p}\right|^{2} d s+\frac{1+\alpha}{2} E \int_{t}^{T}\left|Z_{s}^{n}-Z_{s}^{p}\right|^{2} d s\\
&+&2E \int_{t}^{T}\left(Y_{s}^{n}-V_{s}\right)^{-} d K_{s}^{p}+2 E \int_{t}^{T}\left(Y_{s}^{p}-V_{s}\right)^{-} d K_{s}^{n}.\\
\end{eqnarray*}
\begin{eqnarray*}
\hspace{2mm}E \int_{t}^{T}\left|Z_{s}^{n}-Z_{s}^{p}\right|^{2} d s &\leq& c\Bigg (
E \int_{t}^{T}\left|Y_{s}^{n}-Y_{s}^{p}\right|^{2} d s +
\left(E\left(\sup _{0 \leq t \leq T}\left|\left(Y_{t}^{n}-V_{t}\right)^{-}\right|^{2}\right) \cdot E\left(K_{T}^{p}\right)^{2}\right)^{\frac{1}{2}}\\
&+&\left(E\left(\sup _{0 \leq t \leq T}\left|\left(Y_{t}^{p}-V_{t}\right)^{-}\right|^{2}\right) \cdot E\left(K_{T}^{n}\right)^{2}\right)^{\frac{1}{2}}\Bigg ).\\
\end{eqnarray*}
 According to the Lemma 4.5£¬ we  prove that $E\left(sup_{0\leq t\leq T}\left|(Y_t^{n}-V_t)^{-}\right|^{2}\right)\rightarrow0$, as $n\rightarrow\infty$.\\Hence
\begin{eqnarray}
E \int_{0}^{T}\left(Z_{t}^{n}-Z_{t}^{p}\right)^{2} d t \rightarrow 0,
E \int_{0}^{T}\left(Y_{t}^{n}-Y_{t}^{p}\right)^{2} d t \rightarrow 0, ~~as  ~~n,p\rightarrow\infty.
\end{eqnarray}
Now we begin to prove $Y$ is continuous.\\
\begin{eqnarray*}
&\hspace{2mm}&\left|Y_{t}^{n}-Y_{t}^{p}\right|^{2}+\int_{t}^{T}\left|Z_{s}^{n}-Z_{s}^{p}\right|^{2} d s\\
&=&2 \int_{t}^{T}\left[f\left(s, X_s, Y_{s}^{n}, Z_{s}^{n},V_s\right)-f\left(s, X_s,Y_{s}^{p}, Z_{s}^{p},V_s\right)\right]\left(Y_{s}^{n}-Y_{s}^{p}\right) d s\\
&+&\int_{t}^{T}\left|g\left(s, X_s,Y_{s}^{n}, Z_{s}^{n}\right)-g\left(s, X_s, Y_{s}^{p}, Z_{s}^{p}\right)\right|^{2} d s\\
&+&2 \int_{t}^{T}\left[g\left(s, X_s,Y_{s}^{n}, Z_{s}^{n}\right)-g\left(s, X_s, Y_{s}^{p}, Z_{s}^{p}\right)\right]\left(Y_{s}^{n}-Y_{s}^{p}\right) d B_{s}-2 \int_{t}^{T}\left(Y_{s}^{n}-Y_{s}^{p}\right)\left(Z_{s}^{n}-Z_{s}^{p}\right) d W_{s}\\
&+&2 \int_{t}^{T}\left(Y_{s}^{n}-Y_{s}^{p}\right) d\left(K_{s}^{n}-K_{s}^{p}\right).\\
\end{eqnarray*}
\begin{eqnarray*}
&\hspace{2mm}&\sup _{0 \leq t \leq T}\left|Y_{t}^{n}-Y_{t}^{p}\right|^{2} \leq 2 \int_{0}^{T}\left| f(s,X_s, Y_s^{n}, Z_{s}^{n},V_s £©-f\left(s, X_s,Y^{p}_s, Z_{s}^{p},V_s\right)\right| \cdot\left|Y_{s}^{n}-Y_{s}^{p}\right| d s\\
&+&2 \sup _{0 \leq t \leq T}\left|\int_{t}^{T} g\left(s, X_s,Y^{n}_s, Z_{s}^{n}\right)-g\left(s, X_s,Y^{p}_s, Z_{s}^{p}\right)\left(Y_{s}^{n}-Y_{s}^{p}\right) d B_{s}\right|\\
&+&2 \sup _{0 \leq t \leq T}\left|\int_{t}^{T}\left(Y_{s}^{n}-Y_{s}^{p}\right)\left(Z_{s}^{n}-Z_{s}^{p}\right) d W_{s}\right|+\int_{0}^{T}\left|g\left(s, X_s, Y^{n}_s, Z_{s}^{n}\right)-g\left(s, X_s,Y^{p}_s, Z_{s}^{p}\right)\right|^{2} d s\\
&+&2\int_{0}^{T}\left(Y_{s}^{p}-V_{s}\right)^{-} d K_{s}^{n}+2 \int_{0}^{T}\left(Y_{s}^{n}-S_{s}\right)^{-} d K_{s}^{p}.
\end{eqnarray*}
\text { From Burkholder-Davis-Gundy inequality},
\begin{eqnarray*}
&\hspace{2mm}&E \sup _{0 \leq t \leq T}\left|Y_{t}^{n}-Y_{t}^{p}\right|^{2} \leq \frac{1}{2} E \sup _{0 \leq t \leq T}\left|Y_{t}^{n}-Y_{t}^{p}\right|^{2}+c E \int_{0}^{T}\left(\left|Y_{s}^{n}-Y_{s}^{p}\right|^{2}+\left|Z_{s}^{n}-Z_{s}^{p}\right|^{2}\right) d s\\
&+&\left(E\left[\sup _{0 \leq t \leq T}\left|\left(Y_{t}^{n}-V_{t}\right)^{-}\right|^{2}\right] \cdot E\left|K_{T}^{p}\right|^{2}\right)^{\frac{1}{2}}
+\left(E\left[\sup _{0 \leq t \leq T}\left|\left(Y_{t}^{p}-V_{t}\right)^{-}\right|^{2}\right] \cdot E\left|K_{T}^{n}\right|^{2}\right)^{\frac{1}{2}}.
\end{eqnarray*}
We  get $E\left(\sup _{0 \leq t \leq T}\left|Y_{t}^{n}-Y_{t}^{p}\right|^{2}\right) \rightarrow 0$  as $n, p \rightarrow \infty$.~~$Y^{n}$ convergence uniformly in $t$ to $Y,$ a.s.
hence $Y$ is continuous.\\
In addition, we have denoted that $K^{n}_t$ is a increasing process with $E\left(\left(K_{T}^{n}\right)^{2}\right) \leq C$, it is obvious that $K_{T}<\infty, \text { a.s. }$

\begin{eqnarray*}
&\hspace{2mm}&E(\sup _{0 \leq t \leq T}|K_{t}^{n}-K_{t}^{p}|^{2})\leq c\left.\{E \sup _{0 \leq t \leq T}|Y_{s}^{n}-Y_{s}^{p}|^{2}+E|Y_{0}^{n}-Y_{0}^{p}|^{2}\right.\\
&+&E \int_{0}^{T}(f(s, X_s,Y_{s}^{n}, Z_{s}^{n},V_s)-f(s, X_s, Y_{s}^{p}, Z_{s}^{p},V_s)^{2} d s\\
&+&E(\sup_{0 \leq t \leq T} | \int_{0}^{t} g(s, Y_{s}^{n}, Z_{s}^{n})-g(s, Y_{s}^{p}, Z_{s}^{p}) d B_{s}|)\\
&+&\left.E(\sup _{0 \leq t \leq T}|\int_{0}^{t}(Z_{s}^{n}-Z_{s}^{p}) d W_{s}|)\right.\}.\\
\end{eqnarray*}
From the Lipschitz conditions and the Burkholder-Davis-Gundy inequality, we have
\begin{eqnarray}\label{eq2.84}
E\left(\sup _{0\leq t \leq T}\left(K_{t}^{n}-K_{t}^{p}\right)^{2}\right) \rightarrow 0, \quad \text { as } \quad n, p \rightarrow \infty.
\end{eqnarray}
It remains to check that~$\int_{0}^{T}\left(Y_{t}-V_{t}\right) d K_{t}=0$.\\
According to (\ref{eq2.83}) and (\ref{eq2.84}), we have
$$\int_{0}^{T}\left(Y_{s}^{n}-V_{s}\right) d K_{s}^{n} \longrightarrow \int_{0}^{T}\left(Y_{s}-V_{s}\right) d K_{s}$$
as $n\longrightarrow\infty$. Moreover $Y_{t}\leq V_{t}$,  a.s.\\
we obtain
$$\int_{0}^{T}\left(Y_{s}^{n}-V_{s}\right) d K_{s}^{n}=-n \int_{0}^{T}\left|\left(Y_{s}^{n}-V_{s}\right)^{-}\right|^{2} d s \leq 0 \text {, a.s. }$$
\noindent
Finally, we take the limit of both sides of the equation of (\ref{eq2.85}), then
we have equation(31).
The proof of the uniqueness are derived from the proof of the Proposition 1.6 in the [12].\\
If there exist another solution $K_r^{t^{*},x,v}$and $Z^{t^{*},x,v}$satisfing equation(\ref{eq2.85}), then we apply $I\hat{t}o$ formula to$(y_t-y_t)^{2}\equiv{0}$ on the $[0,T] $ and take expectation
\begin{eqnarray*}
E \int_{t}^{T}\left|Z_{s}^{t, x, v}-Z_{s}^{t, x, v}\right|^{2} d s+E\left[\left(K_{T}^{t, x, v}-K_{T}^{t, x, v}\right)-\left(K_{t}^{t, x, v}-K_{t}^{t, x, v}\right)\right]^{2}=0.
\end{eqnarray*}
therefore $Z^{t, x, v} \equiv Z^{t^{*}, x, v}, K_{r}^{t, x, v} \equiv K_{r}^{t^{*}, x, v}$ for any $t\in[0,T]$.\\
We device that $E \int_{0}^{T}\left(Z_{t}^{n}-Z_{t}^{p}\right)^{2} d t \rightarrow 0$, by the lemma 4.1 in \cite{OT2006}  , we know that
\begin{eqnarray*}
Z_{r}^{t, x, v}=\sigma^{*} \nabla y_{r}^{t, x, v}=\sigma^{*} \nabla V\left(r, x_{r}^{t, x, v}\right).
\end{eqnarray*}
Then it remain to prove that $y_{r}^{t, x, v}=V(r, x_{r}^{t, x, v})$.
From the BDSDE(13), we have
\begin{eqnarray*}
K^{t,x,v;n}_T-K^{t,x,v;n}_t&=&y^{t,x,v;n}_t-V(T,x^{t,x,v}_T)-\int^T_tf(r,x^{t,x,v}_r,y^{t,x,v;n}_r,z^{t,x,v;n}_r,v_r)dr\\
&-&\int^T_tg(r,x^{t,x,v}_r,y^{t,x,v;n}_r,z^{t,x,v;n}_r)dBr+\int^T_tZ^{t,x,v;n}_rdW_r\nonumber\\
&\leq&|y^{t,x,v;n}_t|+|V(T,x^{t,x,v}_T)|+\int^T_tf(r,0,0,0,v_r)dr\nonumber\\
&+&\int^T_t(K|x^{t,x,v}_r|+K|y^{t,x,v;n}_r|+K|Z^{t,x,v;n}_r|)dr+|\int^T_tZ^{t,x,v,n}_rdW_r|\nonumber\\
&+&|\int^T_tg(r,0,0,0)dBr|+\int^T_t(K|x^{t,x,v}_r|+K|y^{t,x,v;n}_r|+K|Z^{t,x,v;n}_r|)dBr\\
&\leq&|V(t,x)|+|V(T,x^{t,x,v}_T)|+\int^T_tf(r,0,0,0,v_r)dr\nonumber\\
&+&\int^T_t(K|x^{t,x,v}_r|+K|y^{t,x,v;1}_r|+K|V(r,x^{t,x,v}_r)|+K|Z^{t,x,v,n}_r|)dr\nonumber\\
&+&\int^T_t(K|x^{t,x,v}_r|+K|y^{t,x,v;1}_r|+K|V(r,x^{t,x,v}_r)|+K|Z^{t,x,v;n}_r|)dBr\\
&+&|\int^T_tg(r,0,0,0)dBr|+|\int^T_tZ^{t,x,v,n}_rdW_r|.
\end{eqnarray*}
because for any  $v_r$, $E\int^T_t|f(r,0,0,0,v_r)|^2dr\leq M$. We
observe that $y^{t,x,v;i}_r$ is dominated by
$|y^{t,x,v;1}_r|+|V(r,x^{t,x,v}_r)|$. From equation (\ref{eq2.85}) we have
\begin{eqnarray*}
&&E|K^{t,x,v;n}_T|^2\nonumber\leq13|V(t,x)|^2+13E|g(x^{t,x,v}_T)|^2+13E\int^T_t|f(r,0,0,0,v_r)|^2dr\nonumber\\
&&+13E\int^T_t(K^2|x^{t,x,v}_r|^2+K^2|y^{t,x,v;1}_r|^2+K^2|V(r,x^{t,x,v}_r)|^2+K^2|Z^{t,x,v,n}_r|^2)dr\nonumber\\
&&+13\int^T_t(K^2|x^{t,x,v}_r|^2+K^2|y^{t,x,v;1}_r|^2+K^2|V(r,x^{t,x,v}_r)|^2+K^2|Z^{t,x,v,n}_r|^2|)dr\\
&&+13\int^T_t|g(r,0,0,0)dr|^2+13E\int^T_t|Z^{t,x,v,i}_r|^2dr.
\end{eqnarray*}
Thus we can define a $C_3(t,T,x,v)$, independent of $n$, such that
\begin{eqnarray}
E|K^{t,x,v;n}_T|^2\leq
C_3(t,T,x,v)+8(K^2+1)E\int^T_t|Z^{t,x,v;n}_r|^2dr.
\end{eqnarray}
On the other hand, we use It$\widehat{o}$'s formula to~$|y^{t,x,v;n}_r|^2$.
\begin{eqnarray*}
&&|y^{t,x,v;n}_t|^2+E\int^T_t|Z^{t,x,v;n}_r|^2dr\nonumber\\
&=&E|V(T,x^{t,x,v}_T)|^2+2E\int^T_ty^{t,x,v;n}_rf(r,x^{t,x,v}_r,y^{t,x,v;n}_r,Z^{t,x,v;n}_r,v_r)dr\\
&+&E\int^T_t|g(r,x^{t,x,v}_r,y^{t,x,v;n}_r,z^{t,x,v;n}_r)|^2dr+2E\int^T_ty^{t,x,v;n}_rdK^{t,x,v;n}_r\nonumber\\
&\leq&E|V(T,x^{t,x,v}_T)|^2\\
&+&2E\int^T_t|y^{t,x,v;n}_r|(|(f(r,0,0,0,v_r)|+K|x^{t,x,v}_r|+K|y^{t,x,v;n}_r|+K|Z^{t,x,v;n}_r|)dr\\
&+&E\int^T_t(|g(r,0,0,0)+K|x^{t,x,v}_r|+K|y^{t,x,v;n}_r|+K|Z^{t,x,v;n}_r|)^{2}dr\\
&+&2E\int^T_ty^{t,x,v;n}_rdK^{t,x,v;n}_r\\
&\leq&E|V(T,x^{t,x,v}_T)|^2+\int^T_t|f(r,0,0,0,v_r)|^2dr+E\int^T_t|y^{t,x,v;n}_r|^2dr\\
&+&E\int^T_tK^2|y^{t,x,v;n}_r|^2+|x^{t,x,v;n}_r|^2dr+E\int^T_t(2K^2+2K)|y^{t,x,v;n}_r|^2+\frac{1}{2}|Z^{t,x,v;n}_r|^2dr\\
&+&E\int^T_t|4|g(r,0,0,0)|^2+4K^2|x^{t,x,v}_r|^2+4K^2|y^{t,x,v;n}_r|^2+4K^2|Z^{t,x,v,n}_r|^2|dr\\ &+&2E[K^{t,x,v;n}_T \sup_{t\leq s\leq T}|y^{t,x,v;n}_s|]\nonumber\\
&\leq&E|V(T,x^{t,x,v}_T)|^2+E\int^T_t|f(r,0,0,0,v_r)|^2dr+(4K^2+1)E\int^T_t|x^{t,x,v}_r|^2dr\\
&+&(7K^2+2K+1)E\int^T_t[|y^{t,x,v;1}_r|^2+|V(r,x^{t,x,v}_r)|^2]dr+(4K^2+\frac{1}{2})E\int^T_t|Z^{t,x,v;n}_r|^2dr\nonumber\\
&+&4E\int^T_t|g(r,0,0,0)|^2dr+\frac{1}{32(K^2+1)}E|K^{t,x,v;n}_T|^2+64(K^2+1)E\sup_{t\leq s\leq
T}[|y^{t,x,v;1}_s|^2+|V(s,x^{t,x,v}_s)|^2].
\end{eqnarray*}
Then we can define a $C_4(t,T,x,v)$ satifying
\begin{eqnarray*}
E\int_{t}^{T}\left|Z_{r}^{t, x, v ; n}\right|^{2} d r \quad \leq \quad C_{4}(t, T, x, v)+\frac{1}{16\left(K^{2}+1\right)} E\left|K_{T}^{t, x, v ; n}\right|^{2}.
\end{eqnarray*}
Then we have
\begin{eqnarray*}
E\left|K_{T}^{t, x, v ; n}\right|^{2} \leq 2 C_{3}(t, T, x, v)+16\left(K^{2}+1\right) C_{4}(t, T, x, v),
\end{eqnarray*}
it follows that
\begin{eqnarray*}
n^{2} \int_{t}^{r}\left(V\left(s, x_{s}^{t, x, v}\right)-y_{s}^{t, x, v ; i}\right)^{2} d s \leq 2 C_{3}(t, T, x, v)+16\left(K^{2}+1\right) C_{4}(t, T, x, v).
\end{eqnarray*}
Let $ n\rightarrow \infty$, we get $y_{r}^{t, x, v}=V\left(r, x_{r}^{t, x, v}\right)$.\\
On the other hand, for any small $\varepsilon>0$, there exists a control $v^{\prime}\in\mathcal{U}$, $V\left(r, x_{r}^{t, x, v^{\prime}}\right)$
satisfying
\begin{eqnarray*}
V\left(s, x_{s}^{t, x, v^{\prime}}\right) \leq Y_{s}^{s, x_{s}^{t, x, v^{\prime}} }+\varepsilon.
\end{eqnarray*}
\end{proof}
\begin{lemma}\label{lemma9} We assume  (H3.4)(H3.5)(H3.7)(H4.1)-(H4.4),
then $V(s,x^{t,x,v'}_s)$ is a g-
supersolution. Same as the proof of lemma 3.6, there exists a unique increasing
process $(A^{t,x,v'}_r)$ with $A^{t,x,v'}_t=0$ and $E[(A^{t,x,v'}_T)^2]<\infty$ such that
$V(s,x^{t,x,v'}_s)$ coincides with the unique solution $y^{t,x,v}_s$
of the BSDE:
\begin{eqnarray}\label{eq2.335}
y^{t,x,v'}_t&=&V(T,x^{t,x,v'}_T)+\varepsilon+\int^T_tf(r,x^{t,x,v'}_r,y^{t,x,v'}_r,Z^{t,x,v'}_r,v'_r)dr\nonumber\\
&&-(K^{t,x,v'}_T-K^{t,x,v'}_t)+\int^T_tg(r,x^{t,x,v'}_r,y^{t,x,v'}_r,Z^{t,x,v'}_r)dBr-\int^T_tZ^{t,x,v'}_rdW_r.
\end{eqnarray}
where $Z^{t,x,v'}_r=\sigma^\ast\nabla V(r,x^{t,x,v'}_r)$ in the
sense of Definition 4.1..
\end{lemma}
\noindent
The proof of Lemma 4.7. is similar to that of Lemma 4.6..
\begin{theorem}\label{th2} Under the assumption (H3.4)(H3.5)(H3.7)(H4.1)-(H4.4), the value function $V(t,x)$ defined in
{\rm(\ref{eq2.76})} is the unique Sobolev solution of the PDE
{\rm(\ref{eq2.77})}.
\end{theorem}
\begin{proof} Existence: In the stochastic recursive optimal control
problem, the value function $V(t,x)$ defined by (\ref{eq2.76})
satisfies  the Bellman's dynamic programming principle. By Lemma
\ref{lemma8} and Lemma \ref{lemma9}, we know that, for any $v\in
\cal{U}$, there have a unique increasing process $A^{t,x,v}_s$,
$V(s,x^{t,x,v}_s)$ satisfy the following BDSDE:
\begin{eqnarray}\label{eq2.336}
V(s,x^{t,x,v}_s)&=&V(T,x^{t,x,v}_T)+\int^T_sf(r,x^{t,x,v}_r,V(r,x^{t,x,v}_r),\sigma^\ast
\nabla V(r,x^{t,x,v}_r),v_r)dr+(K^{t,x,v}_T-K^{t,x,v}_s)\nonumber\\
&+&\int^T_sg(r,x^{t,x,v}_r,V(r,x^{t,x,v}_r),\sigma^\ast
\nabla V(r,x^{t,x,v}_r))dBr-\int^T_s\sigma^\ast\nabla
V(r,x^{t,x,v}_r)dW_r.
\end{eqnarray}
So it follows easily that
\begin{eqnarray}\label{eq2.337}
V(s,x^{t,x,v}_s)&\geq&
V(T,x^{t,x,v}_T)+\int^T_sf(r,x^{t,x,v}_r,V(r,x^{t,x,v}_r),\sigma^\ast
\nabla
V(r,x^{t,x,v}_r),v_r)dr\nonumber\\
&+&\int^T_sg(r,x^{t,x,v}_r,V(r,x^{t,x,v}_r),\sigma^\ast
\nabla V(r,x^{t,x,v}_r))dBr-\int^T_s\sigma^\ast\nabla V(r,x^{t,x,v}_r)dW_r.~~~~~
\end{eqnarray}
On the other hands, for any small $\varepsilon>0$ there exists a control $v'\in
\cal{U}$, such that $V(s,x^{t,x,v'}_s)$ satisfies the following BSDE:
\begin{eqnarray}\label{eq2.338}
V(s,x^{t,x,v'}_s)-\varepsilon&=&
V(T,x^{t,x,v'}_T)+\int^T_sf(r,x^{t,x,v'}_r,V(r,x^{t,x,v'}_r),\sigma^\ast
\nabla
V(r,x^{t,x,v'}_r),v'_r)dr\nonumber\\
&-&(K^{t,x,v'}_T-K^{t,x,v'}_s)+\int^T_sg(r,x^{t,x,v'}_r,V(r,x^{t,x,v'}_r),\sigma^\ast
\nabla V(r,x^{t,x,v'}_r))dBr\nonumber\\
&-&\int^T_s\sigma^\ast\nabla
V(r,x^{t,x,v'}_r)dW_r.
\end{eqnarray}
Then we have
\begin{eqnarray}\label{eq2.339}
&&V(s,x^{t,x,v'}_s)-\varepsilon\leq
V(T,x^{t,x,v'}_T)+\int^T_sf(r,x^{t,x,v'}_r,V(r,x^{t,x,v'}_r),\sigma^\ast
\nabla
V(r,x^{t,x,v'}_r),v'_r)dr\nonumber\\&&+\int^T_sg(r,x^{t,x,v'}_r,V(r,x^{t,x,v'}_r),\sigma^\ast
\nabla V(r,x^{t,x,v'}_r))dBr-\int^T_s\sigma^\ast\nabla
V(r,x^{t,x,v'}_r)dW_r.
\end{eqnarray}
 We can deduce by the equivalence of norm result (Lemma \ref{lemma7})
that $V\in L^2([t,T],H)$. Indeed, in the stochastic recursive
optimal control problem,  the cost function can be regarded as a
solution of BSDE:
\begin{eqnarray*}\label{eq2.340}
Y^{t,x,v}_s&=&h(x^{t,x,v}_T)+\int^T_sf(r,x^{t,x,v}_r,Y^{t,x,v}_r,Z^{t,x,v}_r,v_r)dr+\int^T_tg(r,x^{t,x,v}_r,y^{t,x,v}_r,z^{t,x,v}_r)dBr\\
&-&\int^T_sZ^{t,x,v}_rdW_r.
\end{eqnarray*}
By usual estimates of BSDEs,  (H3.4)(H3.5)(H3.7)(H4.1)-(H4.4) we know
that
\begin{eqnarray}\label{eq2.341}
&&\int_{R^d}E(|Y^{t,x,v}_t|^2+\int^T_t
|Z^{t,x,v}_r|^2dr)\rho (x)dx\nonumber\\
&\leq& K\int_{R^d}E
|h(x^{t,x,v}_T)|^2\rho(x)dx+K\int_{R^d}\int^T_tE|f(r,x^{t,x,v}_r,0,0,v_r)|^2dr\rho(x)dx\nonumber\\
&\leq&KC\int_{R^d}
|h(x)|^2\rho(x)dx+KC\int_{R^d}\int^T_t|f(r,x,0,0,v_r)|^2dr\rho(x)dx\nonumber\\
&\leq&KC\int_{R^d}
|h(x)|^2\rho(x)dx+KC\int_{R^d}\int^T_tC_1dr\rho(x)dx\nonumber\\
&=&KC\int_{R^d} |h(x)|^2\rho(x)dx+(T-t)C_1CK<\infty.
\end{eqnarray}
So for any $v$, $Y^{t,x,v}_r\in H$, where $H$ is a Hilbert space. Note
$V(t,x)=\sup_{v\in\mathcal{U}}Y^{t,x,v}_t$. Since $U$ is compact set of $R^k$
and Lemma 6.2 in \cite{P1997}, we know that
 $V(t,\cdot)$ is also in
$H$.
Next because (\ref{eq2.337}) holds, then for any nonnegative
$\varphi\in C^{\infty}_c(R^d)$, we have
\begin{eqnarray}\label{eq2.342}
&&\int_{R^d}(V(s,x^{t,x,v}_s),\varphi(x))dx\nonumber\\
&\geq&
\int_{R^d}(V(T,x^{t,x,v}_T),\varphi(x))dx+\int_{R^d}\int^T_s(f(r,x^{t,x,v}_r,V(r,x^{t,x,v}_r),\sigma^\ast\nabla V(r,x^{t,x,v}_r),v_r),\varphi(x))drdx\nonumber\\
&+\!&\!\int_{R^d}\!\int^T_s(g(r,x,V,\sigma^\ast
\nabla V(r,x^{t,x,v}_r)),\varphi(x))dBrdx\!-\!\int_{R^d}\!\int^T_s(\sigma^\ast\nabla
V(r,x^{t,x,v}_r),\varphi(x))dW_rdx.
\end{eqnarray}
It turns out that
\begin{eqnarray}\label{eq2.343}
&&\int_{R^d}(V(s,x),\varphi(s,x))dx\nonumber\\
&\geq&
\int_{R^d}(h(x),\varphi(T,x))dx+\int_{R^d}\int^T_s(f(r,x,V(r,x),\sigma^\ast\nabla V(r,x),v_r),\varphi(r,x))drdx\nonumber\\
&+&\!\int_{R^d}\!\int^T_s(g(r,x,V,\sigma^\ast
\nabla V(r,x)),\varphi(r,x))dBrdx-\!\int_{R^d}\!\int^T_s(\sigma^\ast\nabla V(r,x),\varphi(r,x))dW_rdx.
\end{eqnarray}
Furthermore, using Lemma \ref{lemma6}, we have that
\begin{eqnarray}\label{eq2.344}
&&-\int_{R^d}\int^T_s\sigma^\ast\nabla
V(r,x)\varphi(r,x)dW_rdx\nonumber\\
&=&-\int_{R^d}\sum^d_{j=1}\int^T_s(\sum^d_{i=1}\sigma_{i,j}(r,x)\frac{\partial
V}{\partial x_i}(r,x),\varphi(r,x))dW^j_r\nonumber\\
&=&\int_{R^d}\int^T_s(\mathcal
{L}^vV(r,x),\varphi(r,x))dr-\int_{R^d}\int^T_s(V(r,x),\partial_r\varphi(r,x))drdx.
\end{eqnarray}
Taking (\ref{eq2.344}) into (\ref{eq2.343}), we have that
\begin{eqnarray}\label{eq2.345}
&&\int_{R^d}\int^T_s(V(r,x),\partial_r\varphi(r,x))drdx+\int_{R^d}(V(s,x),\varphi(s,x))dx\nonumber\\
&\geq&\int_{R^d}(h(x),\varphi(T,x))dx+\int_{R^d}\int^T_s(f(r,x,V(r,x),\sigma^\ast\nabla
V(r,x),v_r),\varphi(r,x))drdx\nonumber\\
&+&\int_{R^d}\int^T_s(g(r,x,V(r,x),\sigma^\ast
\nabla V),\varphi(r,x))dBrdx+\int_{R^d}\int^T_s(\mathcal {L}^v_rV(r,x),\varphi(r,x))drdx.
\end{eqnarray}
By virtue of the same techniques, because
(\ref{eq2.339}) holds, so for any nonnegative  $\varphi\in
C^{\infty}_c(R^d)$, we take
$\varepsilon=\frac{\varepsilon'}{\int_{R^d}\varphi(x)dx}$, then

\begin{eqnarray}\label{eq2.346}
&&\int_{R^d}(V(s,x^{t,x,v'}_s),\varphi(x))dx-\varepsilon'\nonumber\\
&\leq&
\int_{R^d}(h(x^{t,x,v'}_T),\varphi(x))dx\nonumber+\int_{R^d}\int^T_s(f(r,x^{t,x,v'}_r,V(r,x^{t,x,v'}_r),\sigma^\ast\nabla V(r,x^{t,x,v'}_r),v'_r),\varphi(x))drdx\\
&+&\int_{R^d}\int^T_s(g(r,x^{t,x,v'}_r,V(r,x^{t,x,v'}_r),\sigma^\ast\nabla V(r,x^{t,x,v'}_r)),\varphi(x))dBrdx\nonumber\\
&-&\int_{R^d}\int^T_s(\sigma^\ast\nabla
V(r,x^{t,x,v'}_r),\varphi(x))dW_rdx.
\end{eqnarray}
This is equivalent to
\begin{eqnarray}\label{eq2.347}
&&\int_{R^d}(V(s,x),\varphi(s,x))dx-\varepsilon'\nonumber\\
&\leq&
\int_{R^d}(h(x),\varphi(T,x))dx+\int_{R^d}\int^T_s(f(r,x,V(r,x),\sigma^\ast\nabla V(r,x),v'_r),\varphi(r,x))drdx\nonumber\\
&+\!&\!\!\!\!\!\!\int_{R^d}\!\int^T_s\!\!(g(r,x,V(r,x),\!\sigma^\ast\nabla V(r,x)),\varphi(r,x))dBrdx\!-\!\!\!\int_{R^d}\!\int^T_s\!\!(\sigma^\ast\nabla V(r,x),\!\varphi(r,x))dW_rdx.
            \end{eqnarray}
Taking (\ref{eq2.344}) into (\ref{eq2.347}), we obtain
\begin{eqnarray}\label{eq2.348}
&&\int_{R^d}\int^T_s(V(r,x),\partial_r\varphi(r,x))drdx+\int_{R^d}(V(s,x),\varphi(s,x))dx-\varepsilon'\nonumber\\
&\leq&\int_{R^d}(h(x),\varphi(T,x))dx+\int_{R^d}\int^T_s(f(r,x,V(r,x),\sigma^\ast\nabla
V(r,x),v'_r),\varphi(r,x))drdx\nonumber\\
&+&\!\!\!\!\int_{R^d}\!\int^T_s\!(g(r,x,V(r,x),\sigma^\ast
\nabla V(r,x)),\varphi(r,x))dBrdx+\!\!\int_{R^d}\!\int^T_s\!(\mathcal {L}^{v'}_rV(r,x),\varphi(r,x))drdx.
\end{eqnarray}
Uniqueness:
Let $\overline{V}$ be another solution of the PDE (\ref{eq2.77}). By
 Definition 4.1, one gets that for any $v\in \mathcal {U}$,
\begin{eqnarray}\label{eq2.349}
&&\int_{R^d}\int^T_s(\overline{V}(r,x),\partial_r\varphi(r,x))drdx+\int_{R^d}(\overline{V}(s,x),\varphi(s,x))dx\nonumber\\
&\geq&\int_{R^d}(h(x),\varphi(T,x))dx+\int_{R^d}\int^T_s(f(r,x,\overline{V}(r,x),\sigma^\ast\nabla
\overline{V}(r,x),v_r),\varphi(r,x))drdx\nonumber\\
&+&\int_{R^d}\int^T_s(g(r,x,,\overline{V}(r,x),\sigma^\ast\nabla
\overline{V}(r,x)),\varphi(r,x))dBrdx\nonumber\\
&+&\int_{R^d}\int^T_s(\mathcal
{L}_t^v\overline{V}(r,x),\varphi(r,x))drdx.
\end{eqnarray}
By Lemma 4.5 in \cite{OT2006}, we have
\begin{eqnarray}\label{eq2.350}
&&\int_{R^d}\int^T_s(\overline{V}(r,x),\partial_r\varphi(r,x))drdx\nonumber\\
&=&\int_{R^d}\Sigma^d_{j=1}\int^T_s(\Sigma^d_{i=1}(\sigma_{i,j}\frac{\partial\overline{V}}{\partial
x_i}(r,x),\varphi(r,x)))dW_rdx+\int_{R^d}\int^T_s(L^v_r\overline{V}(r,x),\varphi(r,x))drdx\nonumber\\
&=&\int_{R^d}\int^T_s((\sigma^\ast\nabla\overline{V})(r,x),\varphi(r,x))dW_rdx+\int_{R^d}\int^T_s(L^v_r\overline{V}(r,x),\varphi(r,x))drdx.
\end{eqnarray}
Taking (\ref{eq2.350}) into (\ref{eq2.349}), we get
\begin{eqnarray}
&&\int_{R^d}(\overline{V}(s,x),\varphi(s,x))dx+\int_{R^d}\int^T_s(\sigma^\ast\nabla\overline{V})(r,x)\varphi(r,x)dW_rdx\nonumber\\
&+&\int_{R^d}\int^T_s(L^v_r\overline{V}(r,x),\varphi(r,x))drdx\nonumber\\
&\geq&\int_{R^d}(h(x),\varphi(T,x))dx+\int_{R^d}\int^T_s(f(r,x,\overline{V}(r,x),\sigma^\ast\nabla
\overline{V}(r,x),v_r),\varphi(r,x))drdx\nonumber\\
&+&\int_{R^d}\int^T_s(g(r,x,,\overline{V}(r,x),\sigma^\ast\nabla
\overline{V}(r,x)),\varphi(r,x))dBrdx+\int_{R^d}\int^T_s(L^v_r\overline{V}(r,x),\varphi(r,x))drdx.\nonumber
\end{eqnarray}
So
\begin{eqnarray}\label{eq2.351}
&&\int_{R^d}(\overline{V}(s,x),\varphi(s,x))dx\nonumber\\
&\geq&\int_{R^d}(h(x),\varphi(T,x))dx+\int_{R^d}\int^T_s(f(r,x,\overline{V}(r,x),\sigma^\ast\nabla \overline{V}(r,x),v_r),\varphi(r,x))drdx\nonumber\\
&+&\int_{R^d}\int^T_s(g(r,x,,\overline{V}(r,x),\sigma^\ast\nabla
\overline{V}(r,x)),\varphi(r,x))dBrdx\nonumber\\
&-&\int_{R^d}\int^T_s(\sigma^\ast\nabla\overline{V})(r,x)\varphi(r,x)dW_rdx.
\end{eqnarray}
\noindent
By  Definition 4.1. we also have that, for any small $\varepsilon>0$,
there exists a control $v'\in\mathcal {U}$, we have
\begin{eqnarray}\label{eq2.352}
&&\int_{R^d}\int^T_s(\overline{V}(r,x),\partial_r\varphi(r,x))drdx+\int_{R^d}(\overline{V}(s,x),\varphi(s,x))dx-\varepsilon\nonumber\\
&\leq&\int_{R^d}(h(x),\varphi(T,x))dx+\int_{R^d}\int^T_s(f(r,x,\overline{V}(r,x),\sigma^\ast\nabla
\overline{V}(r,x),v'_r),\varphi(r,x))drdx\nonumber\\
&+&\int_{R^d}\int^T_s(g(r,x,,\overline{V}(r,x),\sigma^\ast\nabla
\overline{V}(r,x)),\varphi(r,x))dBrdx\nonumber\\
&+&\int_{R^d}\int^T_s(\mathcal
{L}^{v'}_r\overline{V}(r,x),\varphi(r,x))drdx.
\end{eqnarray}
Taking (\ref{eq2.350}) into (\ref{eq2.352}), we have
\begin{eqnarray}\label{eq2.353}
&&\int_{R^d}(\overline{V}(s,x),\varphi(s,x))dx-\varepsilon\nonumber\\
&\leq&\int_{R^d}(h(x),\varphi(T,x))dx+\int_{R^d}\int^T_s(f(r,x,\overline{V}(r,x),\sigma^\ast\nabla \overline{V}(r,x),v'_r),\varphi(r,x))drdx\nonumber\\
&+&\int_{R^d}\int^T_s(g(r,x,,\overline{V}(r,x),\sigma^\ast\nabla
\overline{V}(r,x)),\varphi(r,x))dBrdx\nonumber\\
&-&\int_{R^d}\int^T_s(\sigma^\ast\nabla\overline{V})(r,x)\varphi(r,x)dW_rdx.
\end{eqnarray}
Let us make the change of variable $y=\widehat{x}^{t,x,v}_r$ in each
term of (\ref{eq2.351}), then
\begin{eqnarray}\int_{R^d}(\overline{V}(s,x),\varphi(s,x))dx=\int_{R^d}(\overline{V}(s,x^{t,y,v}_s),\varphi(y))dy,\end{eqnarray}
\begin{eqnarray}\int_{R^d}(h(x),\varphi(T,x))dx=\int_{R^d}(h(x^{t,y,v}_T),\varphi(y))dy,\end{eqnarray}
\begin{eqnarray}
&&\int_{R^d}\int^T_s(f(r,x,\overline{V}(r,x),\sigma^\ast\nabla
\overline{V}(r,x),v_r),\varphi(r,x))drdx\nonumber\\
&=&\int_{R^d}\int^T_s(f(r,x^{t,y,v}_r,\overline{V}(r,x^{t,y,v}_r),(\sigma^\ast\nabla
\overline{V})(r,x^{t,y,v}_r),v_r),\varphi(y))drdy,
\end{eqnarray}
\begin{eqnarray}
&&\int_{R^d}\int^T_s(g(r,x,,\overline{V}(r,x),\sigma^\ast\nabla
\overline{V}(r,x)),\varphi(r,x))dBrdx\notag\\
&=&\int_{R^d}\int^T_s(g(r,x^{t,y,v}_r,,\overline{V}(r,x),\sigma^\ast\nabla
\overline{V}(r,x)),\varphi(y))dBrdy.
\end{eqnarray}
\begin{eqnarray}\int_{R^d}\int^T_s((\sigma^\ast\nabla\overline{V})(r,x),\varphi(r,x))dW_rdx=\int_{R^d}\int^T_s((\sigma^\ast\nabla\overline{V})(r,x^{t,y,v}_r),\varphi(y))dW_rdy.
\end{eqnarray}
So (\ref{eq2.351}) becomes
\begin{eqnarray}
&&\int_{R^d}\overline{V}(s,x^{t,y,v}_s)\varphi(y)dy\nonumber\\
&\geq&\int_{R^d}h(x^{t,y,v}_T)\varphi(y)dy\nonumber+\int^T_s\int_{R^d}(f(r,x^{t,y,v}_r,\overline{V}(r,x^{t,y,v}_r),(\sigma^\ast\nabla\overline{V})(r,x^{t,y,v}_r),v_r),\varphi(y))dydr\nonumber\\
&+&\int_{R^d}\int^T_s(g(r,x^{t,y,v}_r,,\overline{V}(r,x),\sigma^\ast\nabla
\overline{V}(r,x)),\varphi(y))dBrdy\nonumber\\
&-&\int^T_s\int_{R^d}((\sigma^\ast\nabla\overline{V})(r,x^{t,y,v}_r),\varphi(y))dydW_r.
\end{eqnarray}
Since $\varphi$ is arbitrary, we have proven that for a.e. $y$,
\begin{eqnarray}\label{eq2.354}
&&\overline{V}(s,x^{t,y,v}_s)\geq
h(x^{t,y,v}_T)+\int^T_sf(r,x^{t,y,v}_r,\overline{V}(r,x^{t,y,v}_r),(\sigma^\ast\nabla\overline{V})(r,x^{t,y,v}_r),v_r)dr\nonumber\\
&&+\int^T_sg(r,x^{t,y,v}_r,\overline{V}(r,x^{t,y,v}_r),(\sigma^\ast\nabla\overline{V})(r,x^{t,y,v}_r))dBr-\int^T_s(\sigma^\ast\nabla\overline{V})(r,x^{t,y,v}_r)dW_r.
\end{eqnarray}
Let
\begin{eqnarray}&&\overline{y}^{t,y,v}_s=h(x^{t,y,v}_T)+\int^T_sf(r,x^{t,y,v}_r,\overline{V}(r,x^{t,y,v}_r),(\sigma^\ast\nabla\overline{V})(r,x^{t,y,v}_r),v_r)dr\nonumber\\
&&+\int^T_sg(r,x^{t,y,v}_r,\overline{V}(r,x^{t,y,v}_r),(\sigma^\ast\nabla\overline{V})(r,x^{t,y,v}_r))dBr-\int^T_s(\sigma^\ast\nabla\overline{V})(r,x^{t,y,v}_r)dW_r.
\end{eqnarray}
Then
\begin{eqnarray}
\overline{V}(s,x^{t,y,v}_s)&=&
h(x^{t,y,v}_T)+\int^T_sf(r,x^{t,y,v}_r,\overline{V}(r,x^{t,y,v}_r),(\sigma^\ast\nabla\overline{V})(r,x^{t,y,v}_r),v_r)dr\nonumber\\
&+&\int^T_sg(r,x^{t,y,v}_r,\overline{V}(r,x^{t,y,v}_r),(\sigma^\ast\nabla\overline{V})(r,x^{t,y,v}_r))dBr+(\overline{V}(s,x^{t,y,v}_s)-\overline{y}^{t,y,v}_s)\notag\\&-&\int^T_s(\sigma^\ast\nabla\overline{V})(r,x^{t,y,v}_r)dW_r.
\end{eqnarray}
Here $\overline{V}(s,x^{t,y,v}_s)-\overline{y}^{t,y,v}_s\geq 0$, so
by the comparison theorem of the BDSDE, we know that the
$\overline{V}(r,x^{t,y,v}_r)$ is the g-supersolution of the BSDE
(\ref{eq2.74}). So we have
\begin{eqnarray}\label{eq2.355}
\overline{V}(s,x^{t,y,v}_s)\geq Y^{s,x^{t,y,v}_s,v}_s.
\end{eqnarray}
Let us make the same change of variable $y=\widehat{x}^{t,x,v'}_r$
in each term of (\ref{eq2.353}), so (\ref{eq2.353}) becomes
\begin{eqnarray}
&&\int_{R^d}\overline{V}(s,x^{t,y,v'}_s)\varphi(y)dy-\varepsilon\nonumber\\
&\leq&\int_{R^d}h(x^{t,y,v'}_T)\varphi(y)dy\nonumber
+\int^T_s\int_{R^d}(f(r,x^{t,y,v'}_r,\overline{V}(r,x^{t,y,v'}_r),(\sigma^\ast\nabla\overline{V})(r,x^{t,y,v'}_r),v'_r),\varphi(y))dydr\nonumber\\
&+&\int^T_s\int_{R^d}(g(r,x^{t,y,v'}_r,\overline{V}(r,x^{t,y,v'}_r),(\sigma^\ast\nabla\overline{V})(r,x^{t,y,v'}_r)),\varphi(y))dydBr\\
&-&\int^T_s\int_{R^d}(\sigma^\ast\nabla\overline{V})(r,x^{t,y,v'}_r)\varphi(y)dydW_r.\nonumber
\end{eqnarray}
Since $\varphi$ is arbitrary, we have proven that for
almost every $y$
\begin{eqnarray}\label{eq2.356}
&&\overline{V}(s,x^{t,y,v'}_s)-\frac{\varepsilon}{\int_{R^d}\varphi(y)dy}\nonumber\\
&\leq&h(x^{t,y,v'}_T)+\int^T_sf(r,x^{t,y,v'}_r,\overline{V}(r,x^{t,y,v'}_r),(\sigma^\ast\nabla\overline{V})(r,x^{t,y,v'}_r),v'_r)dr\nonumber\\
&+&\int^T_sg(r,x^{t,y,v'}_r,\overline{V}(r,x^{t,y,v'}_r),(\sigma^\ast\nabla\overline{V})(r,x^{t,y,v'}_r))dBr-\int^T_s(\sigma^\ast\nabla\overline{V})(r,x^{t,y,v'}_r)dW_r.
\end{eqnarray}
 Let
$\widetilde{V}(s,x^{t,y,v'}_s)=\overline{V}(s,x^{t,y,v'}_s)-\frac{\varepsilon}{\int_{R^d}\varphi(y)dy}$.

\noindent
Then
\begin{eqnarray}\label{eq2.358}
&&\widetilde{V}(s,x^{t,y,v'}_s)+\frac{\varepsilon}{\int_{R^d}\varphi(y)dy}\nonumber\\
&\leq&h(x^{t,y,v'}_T)+\int^T_sf(r,x^{t,y,v'}_r,\widetilde{V}(r,x^{t,y,v'}_r)+\frac{\varepsilon}{\int_{R^d}\varphi(y)dy},(\sigma^\ast\nabla\widetilde{V})(r,x^{t,y,v'}_r),v'_r)dr\nonumber\\
&+&\int^T_sg(r,x^{t,y,v'}_r,\widetilde{V}(r,x^{t,y,v'}_r)+\frac{\varepsilon}{\int_{R^d}\varphi(y)dy},(\sigma^\ast\nabla\widetilde{V})(r,x^{t,y,v'}_r))dBr\notag\\
&-&\int^T_s(\sigma^\ast\nabla\widetilde{V})(r,x^{t,y,v'}_r)dW_r+\frac{\varepsilon}{\int_{R^d}\varphi(y)dy}.
\end{eqnarray}
Define
\begin{eqnarray}
K^{t,y,v'}
&=&h(x^{t,y,v'}_T)+\int^T_sf(r,x^{t,y,v'}_r,\widetilde{V}(r,x^{t,y,v'}_r)+\frac{\varepsilon}{\int_{R^d}\varphi(y)dy},(\sigma^\ast\nabla\widetilde{V})(r,x^{t,y,v'}_r),v'_r)dr\nonumber\\
&+&\int^T_sg(r,x^{t,y,v'}_r,\widetilde{V}(r,x^{t,y,v'}_r)+\frac{\varepsilon}{\int_{R^d}\varphi(y)dy},(\sigma^\ast\nabla\widetilde{V})(r,x^{t,y,v'}_r))dBr\\
&-&\int^T_s(\sigma^\ast\nabla\widetilde{V})(r,x^{t,y,v'}_r)dW_r+\frac{\varepsilon}{\int_{R^d}\varphi(y)dy}.\nonumber
\end{eqnarray}
By  (\ref{eq2.358}), we can know that
\begin{eqnarray}\label{eq2.359}
&&\widetilde{V}(s,x^{t,y,v'}_s)+\frac{\varepsilon}{\int_{R^d}\varphi(y)dy}\nonumber\\
&&=h(x^{t,y,v'}_T)-(K^{t,y,v'}-\widetilde{V}(s,x^{t,y,v'}_s)-\frac{\varepsilon}{\int_{R^d}\varphi(y)dy})\nonumber\\
&&+\int^T_sf(r,x^{t,y,v'}_r,\widetilde{V}(r,x^{t,y,v'}_r)+\frac{\varepsilon}{\int_{R^d}\varphi(y)dy},(\sigma^\ast\nabla\widetilde{V})(r,x^{t,y,v'}_r),v'_r)dr\nonumber\\
&&+\int^T_sg(r,x^{t,y,v'}_r,\widetilde{V}(r,x^{t,y,v'}_r)+\frac{\varepsilon}{\int_{R^d}\varphi(y)dy},(\sigma^\ast\nabla\widetilde{V})(r,x^{t,y,v'}_r))dBr \notag\\
&&-\int^T_s(\sigma^\ast\nabla\widetilde{V})(r,x^{t,y,v'}_r)dW_r+\frac{\varepsilon}{\int_{R^d}\varphi(y)dy}.
\end{eqnarray}
Because
$K^{t,y,v'}-\widetilde{V}(s,x^{t,y,v'}_s)-\frac{\varepsilon}{\int_{R^d}\varphi(y)dy}\geq
0$, so by the comparison theorem of BDSDEs, we knows that
$$\widetilde{V}(s,x^{t,y,v'}_s)+\frac{\varepsilon}{\int_{R^d}\varphi(y)dy}\leq Y^{s,x^{t,y,v'}_s,v'}_s+\frac{\varepsilon}{\int_{R^d}\varphi(y)dy}.$$
So
\begin{eqnarray}\label{eq2.360}
\overline{V}(s,x^{t,y,v'}_s)\leq
Y^{s,x^{t,y,v'}_s,v'}_s+\frac{\varepsilon}{\int_{R^d}\varphi(y)dy}.
\end{eqnarray}
Finally combining (\ref{eq2.355}) and (\ref{eq2.360}), we know that
$$
\overline{V}(t,y)= \sup_{v\in\mathcal {U}}Y^{t,y,v}_t.$$

Thus $\overline{V}(t,y)$ is also the value of $\sup_{v\in\mathcal
{U}}J(t,y,v)$, from uniqueness of the solution of cost functional
and the uniqueness of supremum, we get uniqueness of weak solution
for PDEs (\ref{eq2.77}), i.e. $\overline{V}(t,x)=V(t,x)$.$
$
\end{proof}



\end{document}